%% file: su3flows3.tex
\documentclass{article}
\usepackage{amssymb}
\usepackage[dvips]{graphicx}
\usepackage{epstopdf}
\usepackage{hyperref}

\newtheorem{theorem}{Theorem}[section]

\newtheorem{corollary}[theorem]{Corollary}

\newtheorem{definition}[theorem]{Definition}
\newtheorem{example}[theorem]{Example}

\newtheorem{lemma}[theorem]{Lemma}

\newtheorem{proposition}[theorem]{Proposition}
\newtheorem{remark}[theorem]{Remark}

\newenvironment{proof}[1][Proof]{\noindent\textbf{#1.} }{\ \rule{0.5em}{0.5em}}

\renewcommand{\theequation}{\thesection.\arabic{equation}}
\input{tcilatex}
\input{epsframe}
\begin{document}

\title{Modified Laplacian coflow of $G_{2}$-structures on manifolds with
symmetry}
\author{Sergey Grigorian \\
University of Texas - Pan American\\
1201 W University Drive\\
Edinburg, TX 78539\\
USA}
\date{}
\maketitle

\begin{abstract}
We consider $G_{2}$-structures on $7$-manifolds that are warped products of
an interval and a six-manifold, which is either a Calabi-Yau manifold, or a
nearly K\"{a}hler manifold. We show that in these cases the $G_{2}$%
-structures are determined by their torsion components up to a phase factor.
We then study the modified Laplacian coflow $\frac{d\psi }{dt}=\Delta _{\psi
}\psi +2d\left( \left( C-\func{Tr}T\right) \varphi \right) $ of these $G_{2}$%
-structures, where $\varphi $ and $\psi $ are the fundamental $3$-form and $%
4 $-form which define the $G_{2}$-structure and $\Delta _{\psi }$ is the
Hodge Laplacian associated with the $G_{2}$-structure. This flow is known to
have short-time existence and uniqueness. We analyse the soliton equations
for this flow and obtain new compact soliton solutions.
\end{abstract}

\section{Introduction}

Geometric flows play a very important role in the study of various geometric
objects. Flows of $G_{2}$-structures on $7$-dimensional manifolds have been
first introduced by Robert Bryant in \cite{bryant-2003}. The original
Laplacian flow of $G_{2}$-structures was given by 
\begin{equation}
\frac{\partial \varphi }{\partial t}=\Delta _{\varphi }\varphi 
\label{lapflow1}
\end{equation}%
where $\varphi $ is the $3$-form that defines the $G_{2}$-structure (and
hence the metric), and $\Delta _{\varphi }=dd^{\ast }+d^{\ast }d$ is the
Hodge Laplacian associated with this $G_{2}$-structure. In general, this is
a non-parabolic, non-linear PDE for $\varphi $ \cite{GrigorianCoflow}.
However, if initially $\varphi $ is a \emph{closed}\textbf{\ }(or sometimes
known as \emph{calibrated}) $G_{2}$-structure, that is when $d\varphi =0$,
the flow (\ref{lapflow1}) becomes a flow of closed $G_{2}$-structures (since 
$\Delta _{\varphi }\varphi =dd^{\ast }\varphi $ becomes an exact form), and
acquires much nicer properties. This flow of closed $G_{2}$-structures can
then be interpreted as the gradient flow of Hitchin's volume functional \cite%
{Hitchin:2000jd-arxiv}. Also, as it has been shown in \cite{BryantXu,XuYe},
it has short-time existence and uniqueness, and moreover a stability
property \cite{XuYe}. There has also been work done on related flows of $%
G_{2}$-structures - such as heat flows by Weiss-Witt \cite%
{WeissWitt1,WeissWitt2} and a general overview of flows of $G_{2}$%
-structures by Karigiannis \cite{karigiannis-2007}, as well as the Laplacian
coflow of \emph{co-closed }$G_{2}$-structures, which was introduced by
Karigiannis-McKay-Tsui in \cite{KarigiannisMcKayTsui}. This was a Laplacian
flow $\frac{\partial \psi }{\partial t}=-\Delta _{\psi }\psi $ of the dual $4
$-form $\psi =\ast _{\varphi }\varphi $ which is now assumed to be closed,
so that $\Delta _{\psi }\psi $ is an exact form and preserves the closed
property of $\psi .$ It was later shown by the current author in \cite%
{GrigorianCoflow} that the Laplacian flow of co-closed $G_{2}$-structures is
not even a weakly parabolic flow, and in fact the symbol of the operator $%
\Delta _{\psi }\psi $ has a mixed signature. Therefore, a\emph{\ modified
Laplacian coflow} was introduced in \cite{GrigorianCoflow}:%
\begin{equation}
\frac{d\psi }{dt}=\Delta _{\psi }\psi +2d\left( \left( C-\func{Tr}T\right)
\varphi \right) .  \label{modcoflow}
\end{equation}%
Here $\func{Tr}T$ is the trace of the torsion tensor of the $G_{2}$%
-structure defined by the $4$-form $\psi $ and $C$ is any constant. Note
that if $\psi $ is co-closed, then the right hand side of (\ref{modcoflow})
is exact and hence the flow preserves the cohomology class of $\psi $ in the
same way as the original Laplacian coflow. The flow (\ref{modcoflow}) is
weakly parabolic in the direction of closed forms and short-time existence
and uniqueness of solutions was shown in \cite{GrigorianCoflow}.

In \cite{KarigiannisMcKayTsui} the authors have considered the behavior of
the original coflow on manifolds with symmetry - in particular, where the $7$%
-manifold is a warped product of 1-dimensional space $L$ (either a line
interval or a circle) and a six-dimensional space $N^{6}$ - which is either
a Calabi-Yau or a nearly K\"{a}hler manifold. As originally shown by Ivanov
and Cleyton in \cite{CleytonIvanovCurv}, $G_{2}$-structures on such warped
product manifolds always have a vanishing $14$-dimensional torsion
component, which excludes the possibility of them admitting a
non-torsion-free \emph{closed} $G_{2}$-structure. However, the $7$%
-dimensional torsion component can be set to zero, leaving only the
symmetric part of the torsion tensor, and this gives a co-closed $G_{2}$%
-structure. Therefore, it is natural to study flows of co-closed $G_{2}$%
-structures on these manifolds. In this case, the complicated PDEs from the
general case then become more manageable since the spatial part of the
equation reduces to a 1-dimensional problem on $L$. Furthermore, if soliton
solutions are considered, the equations reduce to a system of ODEs which can
be solved explicitly in some cases. In this paper we will produce a similar
analysis, but for the modified flow (\ref{modcoflow}).

The outline of the paper is following. In Section \ref{secg2struct} we give
an introduction to $G_{2}$-structures and torsion and in Section \ref%
{secsu3struct} we specialize to the case of the warped product manifold $%
N^{6}\times L$. We rederive the expression for the torsion components of a $%
G_{2}$-structure on such a product manifold and give an expression for the
full torsion tensor. In particular, we also show that generically, in this
case, the torsion components in fact determine the $G_{2}$-structure up to a phase factor. We
express the torsion in terms of parameters $\alpha ,\beta ,\gamma $ and show
how they give the $G_{2}$-structure. The parameters $\alpha $ and $\beta $
determine the $\mathbf{1\oplus 27}$ part of the torsion, while $\gamma $
determines the $\mathbf{7}$ part of the torsion. The restriction to
co-closed $G_{2}$-structures is then equivalent to setting $\gamma =0$. In
Section \ref{seclap}, we then express the Laplacian of the $G_{2}$-structure
in terms of $\alpha ,\beta ,\gamma $ and in Section \ref{secflow} we derive
the modified Laplacian coflow (\ref{modcoflow}) for the warped product $%
G_{2} $-structure. In the special case where $N^{6}$ is a Calabi-Yau
manifold, and when $C=0$ in (\ref{modcoflow}) we find explicit separable
solutions, which, as expected, blow up in finite time. Finally, in Section %
\ref{secsoliton}, we specialize to soliton solutions of (\ref{modcoflow}).
The equations that we obtain are a system of three nonlinear first order
ODEs, which are similar to the third order nonlinear ODE obtained in \cite%
{KarigiannisMcKayTsui} for the standard Laplacian coflow. However, due to
the extra freedom that we get by including the constant $C$ in (\ref%
{modcoflow}), we are able to obtain new non-trivial solutions. In
particular, in the case when $N^{6}$ is a Calabi-Yau manifold, we obtain
explicit solutions that are periodic and are thus defined when $L\cong
S^{1}. $ Hence these are compact soliton solutions. In the more complicated
case when $N^{6}$ is nearly K\"{a}hler, the equations are still very
difficult to analyze, however we systematically consider solutions where at
least one of the dependent variables is constant. This way we recover some
of the solutions given in \cite{KarigiannisMcKayTsui}, as well as new
solutions.

\section{$G_{2}$-structures and torsion}

\setcounter{equation}{0}\label{secg2struct}The 14-dimensional group $G_{2}$
is the smallest of the five exceptional Lie groups and is closely related to
the octonions. In particular, $G_{2}$ can be defined as the automorphism
group of the octonion algebra. Taking the imaginary part of octonion
multiplication of the imaginary octonions defines a vector cross product on $%
V=\mathbb{R}^{7}$ and the group that preserves the vector cross product is
precisely $G_{2}$. A more detailed account of the relationship between
octonions and $G_{2}$ can be found in \cite{BaezOcto, GrigorianG2Review}.The
structure constants of the vector cross product define a $3$-form on $%
\mathbb{R}^{7}$, hence $G_{2}$ can alternatively be defined as the subgroup
of $GL\left( 7,\mathbb{R}\right) $ that preserves a particular $3$-form \cite%
{Joycebook}. In general, given an $n$-dimensional manifold $M$, a $G$%
-structure on $M$ for some Lie subgroup $G$ of $GL\left( n,\mathbb{R}\right) 
$ is a reduction of the frame bundle $F$ over $M$ to a principal subbundle $%
P $ with fibre $G$. A $G_{2}$-structure is then a reduction of the frame
bundle on a $7$-dimensional manifold $M$ to a $G_{2}$ principal subbundle.
It turns out that there is a $1$-$1$ correspondence between $G_{2}$%
-structures on a $7$-manifold and smooth $3$-forms $\varphi $ for which the $%
7$-form-valued bilinear form $B_{\varphi }$ as defined by (\ref{Bphi}) is
positive definite (for more details, see \cite{Bryant-1987} and the arXiv
version of \cite{Hitchin:2000jd}). 
\begin{equation}
B_{\varphi }\left( u,v\right) =\frac{1}{6}\left( u\lrcorner \varphi \right)
\wedge \left( v\lrcorner \varphi \right) \wedge \varphi  \label{Bphi}
\end{equation}%
Here the symbol $\lrcorner $ denotes contraction of a vector with the
differential form: 
\[
\left( u\lrcorner \varphi \right) _{mn}=u^{a}\varphi _{amn}. 
\]%
Note that we will also use this symbol for contractions of differential
forms using the metric.

A smooth $3$-form $\varphi $ is said to be \emph{positive }if $B_{\varphi }$
is the tensor product of a positive-definite bilinear form and a
nowhere-vanishing $7$-form. In this case, it defines a unique metric $%
g_{\varphi }$ and volume form $\mathrm{vol}$ such that for vectors $u$ and $%
v $, the following holds 
\begin{equation}
g_{\varphi }\left( u,v\right) \mathrm{vol}=\frac{1}{6}\left( u\lrcorner
\varphi \right) \wedge \left( v\lrcorner \varphi \right) \wedge \varphi
\label{gphi}
\end{equation}

In components we can rewrite this as 
\begin{equation}
\left( g_{\varphi }\right) _{ab}=\left( \det s\right) ^{-\frac{1}{9}}s_{ab}\ 
\text{where \ }s_{ab}=\frac{1}{144}\varphi _{amn}\varphi _{bpq}\varphi _{rst}%
\hat{\varepsilon}^{mnpqrst}.  \label{metricdefdirect}
\end{equation}%
Here $\hat{\varepsilon}^{mnpqrst}$ is the alternating symbol with $\hat{%
\varepsilon}^{12...7}=+1$. Following Joyce (\cite{Joycebook}), we will adopt
the following definition

\begin{definition}
The pair $\left( \varphi ,g\right) $ for a positive $3$-form $\varphi $ and
corresponding metric $g$ defined by (\ref{gphi}) will be referred to as a $%
G_{2}$-structure.
\end{definition}

\begin{definition}
Given a $G_{2}$-structure $\left( \varphi ,g\right) $, define the Hodge star 
$\ast _{\varphi }$ that is associated with $\left( \varphi ,g\right) $, the
dual $4$-form $\psi =\ast _{\varphi }\varphi $ and the Laplacian $\Delta
_{\varphi }$
\end{definition}

Note that up to an overall sign of the orientation, a $G_{2}$-structure can
alternatively be defined using the $4$-form $\psi .$ In this case, we will
say that $\left( \psi ,g\right) $ is a $G_{2}$-structure giving the Hodge
star $\ast _{\psi }$ and the Laplacian $\Delta _{\psi }$.

Given a $G_{2}$-structure, the spaces of differential forms decompose
orthogonally according to irreducible representation of $G_{2}.$ In
particular, $2$-forms split as $\Lambda ^{2}=\Lambda _{7}^{2}\oplus \Lambda
_{14}^{2}$, where 
\begin{eqnarray*}
\Lambda _{7}^{2} &=&\left\{ \alpha \lrcorner \varphi \text{: for a vector
field }\alpha \right\} \\
\Lambda _{14}^{2} &=&\left\{ \omega \in \Lambda ^{2}\text{: }\left( \omega
_{ab}\right) \in \mathfrak{g}_{2}\right\} =\left\{ \omega \in \Lambda ^{2}%
\text{: }\omega \lrcorner \varphi =0\right\}
\end{eqnarray*}%
Using Hodge duality, a similar decomposition exists for $5$-forms.

The $3$-forms decompose as $\Lambda ^{3}=\Lambda _{1}^{3}\oplus \Lambda
_{7}^{3}\oplus \Lambda _{27}^{3}$, where the one-dimensional component
consists of forms proportional to $\varphi $, forms in the $7$-dimensional
component are defined by a vector field $\Lambda _{7}^{3}=\left\{ \alpha
\lrcorner \psi \text{: for a vector field }\alpha \right\} $, and forms in
the $27$-dimensional component are defined by traceless, symmetric matrices: 
\begin{equation}
\Lambda _{27}^{3}=\left\{ \chi \in \Lambda ^{3}:\chi _{abc}=\mathrm{i}%
_{\varphi }\left( h\right) =h_{[a}^{d}\varphi _{bc]d}^{{}}\text{ for }h_{ab}~%
\text{traceless, symmetric}\right\} .  \label{lam327}
\end{equation}%
Again, by Hodge duality a similar decomposition exists for $4$-forms. A
detailed description of these representations is given in \cite%
{Bryant-1987,bryant-2003}.

The \emph{intrinsic torsion }of a $G_{2}$-structure is defined by $\nabla
\varphi $, where $\nabla $ is the Levi-Civita connection for the metric $g$
that is defined by $\varphi $. Following \cite{karigiannis-2007}, it is easy
to see 
\begin{equation}
\nabla \varphi \in \Lambda _{7}^{1}\otimes \Lambda _{7}^{3}\cong W.
\label{torsphiW}
\end{equation}%
Here we define $W$ as the space $\Lambda _{7}^{1}\otimes \Lambda _{7}^{3}$.
Given (\ref{torsphiW}), we can write 
\begin{equation}
\nabla _{a}\varphi _{bcd}=T_{a}^{\ e}\psi _{ebcd}^{{}}  \label{codiffphi}
\end{equation}%
where $T_{ab}$ is the \emph{full torsion tensor}. We can also invert (\ref%
{codiffphi}) to get an explicit expression for $T$ 
\begin{equation}
T_{a}^{\ m}=\frac{1}{24}\left( \nabla _{a}\varphi _{bcd}\right) \psi ^{mbcd}.
\end{equation}%
This $2$-tensor fully defines $\nabla \varphi $ since pointwise, it has 49
components, and the space $W$ is also 49-dimensional (pointwise). In general
we can split $T_{ab}$ according to representations of $G_{2}$ into \emph{%
torsion components}: 
\begin{equation}
T=\tau _{1}g+\tau _{7}\lrcorner \varphi +\tau _{14}+\tau _{27}
\end{equation}%
where $\tau _{1}$ is a function, and gives the $\mathbf{1}$ component of $T$%
. We also have $\tau _{7}$, which is a $1$-form and hence gives the $\mathbf{%
7}$ component, and, $\tau _{14}\in \Lambda _{14}^{2}$ gives the $\mathbf{14}$
component and $\tau _{27}$ is traceless symmetric, giving the $\mathbf{27}$
component. Hence we can split $W$ as 
\begin{equation}
W=W_{1}\oplus W_{7}\oplus W_{14}\oplus W_{27}.
\end{equation}%
As it was originally shown by Fern\'{a}ndez and Gray \cite{FernandezGray},
there are in fact a total of 16 torsion classes of $G_{2}$-structures that
arise as the $G_{2}$-invariant subspaces of $W$ to which $\nabla \varphi $
belongs. Moreover, as shown in \cite{karigiannis-2007}, the torsion
components relate directly to the expression for $d\varphi $ and $d\psi $.
In fact, in our notation, 
\begin{subequations}%
\label{dptors} 
\begin{eqnarray}
d\varphi &=&4\tau _{1}\psi -3\tau _{7}\wedge \varphi -3\ast \mathrm{i}%
_{\varphi }\left( \tau _{27}\right)  \label{dphi} \\
d\psi &=&-4\tau _{7}\wedge \psi -2\ast \tau _{14}.  \label{dpsi}
\end{eqnarray}%
\end{subequations}%
Note that in the literature (\cite{bryant-2003,CleytonIvanovConf}, for
example) a slightly different convention for torsion components is sometimes
used. Our $\tau _{1}$ component corresponds to $\frac{1}{4}\tau _{0}$, $\tau
_{7}$ corresponds to $-\tau _{1}$ in their notation, $\tau _{14}$
corresponds to $\frac{1}{2}\tau _{2}$ and\ $\mathrm{i}_{\varphi }\left( \tau
_{27}\right) $ corresponds to $-\frac{1}{3}\tau _{3}$ . Similarly, our
torsion classes $W_{1}\oplus W_{7}\oplus W_{14}\oplus W_{27}$ correspond to $%
W_{0}\oplus W_{1}\oplus W_{2}\oplus W_{3}$. In our notation the subscripts
denote the dimensionally of the representation, while in the alternative
notation the subscripts denote the degree of the corresponding differential
form. Also the constant factors are different because we consider the $\tau
_{i}$ as components of the full torsion tensor $T$, while in the alternative
point of view they are regarded as components of the differential forms $%
d\varphi $ and $d\psi $.

\begin{definition}
A $G_{2}$-structure is said to be \emph{torsion-free} if $T=0.$
Equivalently, $d\varphi =0$ and $d\psi =0$ \cite{FernandezGray}.
\end{definition}

\begin{definition}
A $G_{2}$-structure is said to be \emph{closed }if $d\varphi =0.$
Equivalently, $T=\tau _{14}$.
\end{definition}

\begin{definition}
A $G_{2}$-structure is said to be \emph{co-closed }if $d\psi =0$.
Equivalently $T=\tau _{1}g+\tau _{27}$, that is, the skew-symmetric part of $%
T$ vanishes, and the tensor $T$ is thus fully symmetric.
\end{definition}

Note that sometimes closed and co-closed $G_{2}$-structures are called \emph{%
calibrated }and \emph{cocalibrated}, respectively.

\begin{example}
A special case of a co-closed $G_{2}$-structure occurs when $\tau _{27}=0$.
In this case, we have $d\varphi =4\tau _{1}\psi $ with $\tau _{1}$ constant. 
$G_{2}$-structures of this type are called \emph{nearly parallel}, and they
are Einstein manifolds, with $Ric=6\tau _{1}^{2}g$. In particular, the round
sphere $S^{7}$ admits a nearly parallel $G_{2}$-structure \cite%
{FriedrichNPG2sp}.
\end{example}

\section{$G_{2}$-structures on warped product manifolds}

\setcounter{equation}{0}\label{secsu3struct}Consider an $SU\left( 3\right) $%
-structure $\left( g_{6},\omega ,\Omega \right) $ on a $6$-manifold $N^{6}$
- where $g_{6}$ is a Riemannian metric, $\omega $ is a compatible Hermitian
form of type $\left( 1,1\right) $, and $\Omega $ is a nowhere vanishing
smooth complex-valued $3$-form of type $\left( 3,0\right) .$ The $SU\left(
3\right) $-structure forms $\omega $ and $\Omega $ satisfy algebraic
constraints 
\begin{subequations}%
\label{lombom} 
\begin{eqnarray}
\Omega \wedge \omega &=&\bar{\Omega}\wedge \omega =0 \\
\omega ^{3} &=&6\func{vol}_{6} \\
\Omega \wedge \bar{\Omega} &=&-8i\func{vol}_{6}
\end{eqnarray}%
\end{subequations}%
Since we are restricting our attention to Calabi-Yau and nearly K\"{a}hler $%
6 $-manifolds, the exterior derivatives of $\omega $ and $\Omega $ satisfy
the following relations 
\begin{subequations}%
\label{su3struct}%
\begin{eqnarray}
d\omega &=&-3\lambda \func{Re}\Omega \\
&=&-\frac{3}{2}\lambda \Omega -\frac{3}{2}\lambda \bar{\Omega}  \label{dlom}
\\
d\Omega &=&2\lambda i\omega ^{2}  \label{dbom} \\
d\left( \omega ^{2}\right) &=&0  \label{dlomsq}
\end{eqnarray}%
\end{subequations}
where $\lambda $ is a constant. Note that $\lambda =0$ corresponds to a
Calabi-Yau manifold, and generally we will set $\lambda =1$ for a nearly K%
\"{a}hler manifold. The Ricci curvature $Ric_{6}$ of a nearly K\"{a}hler
manifold of type $\lambda $ is given by 
\begin{equation}
Ric_{6}=5\lambda g_{6}  \label{NKRic}
\end{equation}%
More details on nearly K\"{a}hler manifolds are given in \cite%
{GrayNK,MoroianuNagySemmelmann}.

Now suppose $L$ is a $1$-dimensional manifold and consider $%
M^{7}=N^{6}\times L$ with $r$ a local coordinate on $L$. An induced $G_{2}$%
-structure on $M^{7}$ is given by 
\begin{subequations}%
\label{G2prod}%
\begin{eqnarray}
\varphi &=&\func{Re}\Omega +dr\wedge \omega \\
\psi &=&\frac{1}{2}\omega ^{2}+\func{Im}\Omega \wedge dr \\
g_{7} &=&dr^{2}+g_{6}
\end{eqnarray}

\end{subequations}%
More generally, let $F\left( r\right) $ be a smooth, nowhere-vanishing
complex-valued function on $L$ and let $G\left( r\right) $ be a smooth,
real, everywhere positive function on $L$. Then, following \cite%
{ChiossiSalamon, KarigiannisMcKayTsui} we get a warped product $G_{2}$%
-structure on $M^{7}$ given by 
\begin{subequations}%
\label{G2warpedstruct} 
\begin{eqnarray}
\varphi &=&\func{Re}\left( F^{3}\Omega \right) +G\left\vert F\right\vert
^{2}dr\wedge \omega \\
\psi &=&\frac{1}{2}\left\vert F\right\vert ^{4}\omega ^{2}+\func{Im}\left(
F^{3}\Omega \right) \wedge Gdr \\
g_{7} &=&G^{2}dr^{2}+\left\vert F\right\vert ^{2}g_{6} \\
\func{vol}_{7} &=&G\left\vert F\right\vert ^{6}\func{vol}_{6}\wedge dr
\end{eqnarray}%
\end{subequations}
We can write%
\begin{equation}
F=he^{i\frac{\theta }{3}}
\end{equation}%
Then, the $\varphi ,\psi ,$ $g_{7}$ and $\func{vol}_{7}$ from (\ref%
{G2warpedstruct}) can be rewritten as 
\begin{subequations}%
\label{G2warpedstruct2}%
\begin{eqnarray}
\varphi &=&\frac{1}{2}F^{3}\Omega +\frac{1}{2}\bar{F}^{3}\bar{\Omega}%
+Gh^{2}dr\wedge \omega  \label{g2phi} \\
\psi &=&\frac{1}{2}h^{4}\omega ^{2}-\frac{iGF^{3}}{2}\Omega \wedge dr+\frac{%
iG\bar{F}^{3}}{2}\bar{\Omega}\wedge dr \\
g_{7} &=&G^{2}dr^{2}+h^{2}g_{6}  \label{g7met} \\
\func{vol}_{7} &=&Gh^{6}\func{vol}_{6}\wedge dr  \label{vol7met}
\end{eqnarray}%
\end{subequations}
Using the expressions for the metric $g_{7}$ and the volume form $\func{vol}%
_{7}$, we obtain that if $\alpha $ is a $k$-form on $N^{6},$ then 
\begin{subequations}%
\label{warpedhstar} 
\begin{eqnarray}
\ast _{7}\alpha &=&\left( -1\right) ^{k}h^{6-2k}Gdr\wedge \ast _{6}\alpha \\
\ast _{7}\left( dr\wedge \alpha \right) &=&h^{6-2k}G^{-1}\ast _{6}\alpha
\end{eqnarray}%
\end{subequations}%
From these expression we obtain the following useful formulae.

\begin{corollary}[\protect\cite{KarigiannisMcKayTsui}]
Given the metric $g_{7}$ and the volume form $\func{vol}_{7}$ as in (\ref%
{g7met}) and (\ref{vol7met}), the Hodge duals of the $SU\left( 3\right) $%
-equivariant differential forms on $M^{7}$ are:%
\begin{subequations}%
\label{hodgeduals}%
\begin{eqnarray}
\ast _{7}\omega &=&\frac{1}{2}h^{2}Gdr\wedge \omega ^{2} \\
\ast _{7}\Omega &=&iGdr\wedge \Omega \\
\ast _{7}\left( \omega ^{2}\right) &=&2h^{-2}Gdr\wedge \omega \\
\ast _{7}\left( Gdr\wedge \omega \right) &=&\frac{1}{2}h^{2}\omega ^{2} \\
\ast _{7}\left( Gdr\wedge \Omega \right) &=&-i\Omega \\
\ast _{7}\left( Gdr\wedge \omega ^{2}\right) &=&2h^{-2}\omega
\end{eqnarray}%
\end{subequations}%
\end{corollary}

As noted in \cite{KarigiannisMcKayTsui}, we can always redefine the $r$
coordinate in order to set $G=1\,.$ However when looking at a flow of $G_{2}$%
-structures, the function $G\left( r\right) $ will be time-dependent and
hence the reparametrization of $r$. Therefore, in a time-dependent picture
it is convenient to keep $G\left( r\right) $ unrestricted.

Any $3$-form $\chi $ on $M^{7}$ that respects the symmetry of the manifold
must be a linear combination of $\Omega $, $\bar{\Omega}$ and $dr\wedge
\omega $. Therefore, in general, we can write such a $3$-form as 
\begin{equation}
\chi =\frac{1}{2}AF^{3}\Omega +\frac{1}{2}\bar{A}\bar{F}^{3}\bar{\Omega}%
+Gh^{2}Bdr\wedge \omega  \label{gen3form}
\end{equation}%
where $A$ is a smooth complex-valued function on $L$ and $B$ is a smooth
real-valued function on $L.$ Hence, such $3$-forms are uniquely defined by
three real-valued functions on $L$: $\func{Re}A,$ $\func{Im}A$ and $B$. For
a $3$-form given by (\ref{gen3form}), let us use the following notation 
\begin{subequations}%
\label{chinotation} 
\begin{eqnarray}
\func{Re}_{1}\chi &=&\func{Re}A \\
\func{Im}_{1}\chi &=&\func{Im}A \\
\func{Re}_{2}\chi &=&B
\end{eqnarray}%
\end{subequations}
Such a decomposition of $\chi $ effectively gives a decomposition according
to representations of $SU\left( 3\right) $ using the underlying $SU\left(
3\right) $-structure, and in many cases it will be more convenient to use
this decomposition rather than the decomposition according representations
of $G_{2}$ that comes from the $G_{2}$-structure (\ref{G2warpedstruct2}).
However, both will play a role, and it will be necessary to convert between
the two pictures. The $G_{2}$-decomposition of $\chi $ is given by \cite%
{bryant-2003, GrigorianCoflow, karigiannis-2007}:%
\begin{equation}
\chi =X\lrcorner \psi +\mathrm{i}_{\varphi }\left( s\right)  \label{chirep}
\end{equation}%
where $X$ is a vector field given by 
\begin{equation}
X^{\flat }=\frac{1}{4}\ast _{7}\left( \chi \wedge \varphi \right)
\label{Xvect}
\end{equation}%
and which defines the $\Lambda _{7}^{3}$ component of $\chi $, and $s$ is a
symmetric $2$-tensor. The trace part of $s$ gives the $\Lambda _{1}^{3}$
component of $\chi $ and the traceless part defines the $\Lambda _{27}^{3}$
component. The trace of $h$ is given by 
\begin{equation}
\func{Tr}h=\ast _{7}\left( \chi \wedge \psi \right)  \label{trh}
\end{equation}

\begin{proposition}
\label{PropChiG2cpts}Suppose $\chi $ is an $SU\left( 3\right) $-equivariant $%
3$-form on $M^{7}.$ Then using the notation in (\ref{chinotation}), the $%
G_{2}$-decomposition of $\chi $ is $\chi =X\lrcorner \psi +\mathrm{i}%
_{\varphi }\left( s\right) $ where%
\begin{subequations}%
\label{chicpts}%
\begin{eqnarray}
X &=&\left( \func{Im}_{1}\chi \right) G^{-1}\frac{\partial }{\partial r} \\
s &=&\left( 3\func{Re}_{2}\chi -2\func{Re}_{1}\chi \right)
G^{2}dr^{2}+\left( \func{Re}_{1}\chi \right) h^{2}g_{6} \\
\func{Tr}s &=&3\func{Re}_{2}\chi +4\func{Re}_{1}\chi
\end{eqnarray}%
\end{subequations}%
Also note that $s$ with one raised index, denoted by $s^{\sharp }$,\ is
given by 
\begin{equation}
s^{\sharp }=\func{diag}\left( \func{Re}_{2}\chi -2\func{Re}_{1}\chi ,\left( 
\func{Re}_{1}\chi \right) \delta _{6}\right)  \label{hsharp}
\end{equation}
\end{proposition}

\begin{proof}
Let 
\[
\chi =\frac{1}{2}AF^{3}\Omega +\frac{1}{2}\bar{A}\bar{F}^{3}\bar{\Omega}%
+Gh^{2}Bdr\wedge \omega 
\]%
First let us find $\func{Tr}s$ using (\ref{trh}). We thus have 
\begin{eqnarray*}
\chi \wedge \psi &=&\left( \frac{1}{2}AF^{3}\Omega +\frac{1}{2}\bar{A}\bar{F}%
^{3}\bar{\Omega}+Gh^{2}Bdr\wedge \omega \right) \wedge \\
&&\left( \frac{1}{2}h^{4}\omega ^{2}-\frac{iGF^{3}}{2}\Omega \wedge dr+\frac{%
iG\bar{F}^{3}}{2}\bar{\Omega}\wedge dr\right) \\
&=&\frac{1}{4}iGAh^{6}\Omega \wedge \bar{\Omega}\wedge dr-\frac{1}{4}iG\bar{A%
}h^{6}\bar{\Omega}\wedge \Omega \wedge dr+\frac{1}{2}Gh^{6}Bdr\wedge \omega
^{3} \\
&=&2GAh^{6}\func{vol}_{6}\wedge dr+2G\bar{A}h^{6}\func{vol}_{6}\wedge
dr+3Gh^{6}B\func{vol}_{6}\wedge dr \\
&=&\left( 3B+2\left( A+\bar{A}\right) \right) \func{vol}_{7}
\end{eqnarray*}%
where we have used the properties (\ref{lombom}). Hence indeed, 
\begin{eqnarray*}
\func{Tr}s &=&3B+2\left( A+\bar{A}\right) \\
&=&3\func{Re}_{2}\chi +4\func{Re}_{1}\chi .
\end{eqnarray*}%
Now work out $X$ using (\ref{Xvect}). Working out $\chi \wedge \varphi $
using (\ref{lombom}) we get%
\begin{eqnarray*}
\chi \wedge \varphi &=&\left( \frac{1}{2}AF^{3}\Omega +\frac{1}{2}\bar{A}%
\bar{F}^{3}\bar{\Omega}+Gh^{2}Bdr\wedge \omega \right) \wedge \\
&&\left( \frac{1}{2}F^{3}\Omega +\frac{1}{2}\bar{F}^{3}\bar{\Omega}%
+Gh^{2}dr\wedge \omega \right) \\
&=&\frac{1}{4}Ah^{6}\Omega \wedge \bar{\Omega}+\frac{1}{4}\bar{A}h^{6}\bar{%
\Omega}\wedge \Omega \\
&=&\left( -2i\left( A-\bar{A}\right) \right) h^{6}\func{vol}_{6}
\end{eqnarray*}%
Take the Hodge star:%
\begin{eqnarray*}
\ast _{7}\left( \chi \wedge \varphi \right) &=&\left( -2i\left( A-\bar{A}%
\right) \right) Gdr \\
&=&4\left( \func{Im}A\right) Gdr
\end{eqnarray*}%
Thus, indeed, 
\begin{eqnarray*}
X &=&\frac{1}{4}\left( \ast _{7}\left( \chi \wedge \varphi \right) \right)
^{\sharp } \\
&=&\left( \func{Im}A\right) G^{-1}\frac{\partial }{\partial r} \\
&=&\left( \func{Im}_{1}\chi \right) G^{-1}\frac{\partial }{\partial r}
\end{eqnarray*}%
To find $h$, consider the projection of $\chi $ onto $\Lambda _{1}^{3}\oplus
\Lambda _{27}^{2}:$%
\[
\pi _{1\oplus 27}\chi =\chi -X\lrcorner \psi . 
\]%
Thus, 
\begin{eqnarray*}
X\lrcorner \psi &=&\left( \func{Im}A\right) G^{-1}\left( \frac{\partial }{%
\partial r}\lrcorner \psi \right) \\
&=&\frac{1}{2}i\left( \func{Im}A\right) F^{3}\Omega -\frac{1}{2}i\left( 
\func{Im}A\right) \bar{F}^{3}\bar{\Omega}
\end{eqnarray*}%
Therefore, 
\begin{eqnarray*}
\pi _{1\oplus 27}\chi &=&\chi -X\lrcorner \psi =\frac{1}{2}\left( \func{Re}%
A+i\func{Im}A\right) F^{3}\Omega +\frac{1}{2}\left( \func{Re}A-i\func{Im}%
A\right) \bar{F}^{3}\bar{\Omega}+Gh^{2}Bdr\wedge \omega \\
&&-\left( \frac{1}{2}i\left( \func{Im}A\right) F^{3}\Omega -\frac{1}{2}%
i\left( \func{Im}A\right) \bar{F}^{3}\bar{\Omega}\right) \\
&=&\frac{1}{2}\left( \func{Re}A\right) F^{3}\Omega +\frac{1}{2}\left( \func{%
Re}A\right) \bar{F}^{3}\bar{\Omega}+Gh^{2}Bdr\wedge \omega
\end{eqnarray*}%
Now assume without loss of generality that $\pi _{7}\chi =0$, so that $A=%
\func{Re}A$. Also recall that \cite{GrigorianG2Torsion1} 
\[
\chi _{bc(a}\varphi _{d)}^{\ bc}=\frac{4}{3}h_{ad}+\frac{2}{3}\left( \func{Tr%
}h\right) g_{ad} 
\]%
Comparing $\chi $ and $\varphi $ we see that the only non-zero contractions
that involve $\func{Re}A$ will be proportional to $g_{6}$, while the only
non-zero contractions that are proportional to $dr^{2}$ only involve $B.$
Moreover, since $s$ is real, we can in general write%
\[
s+\frac{1}{2}\left( \func{Tr}s\right) g_{7}=c_{2}BG^{2}dr^{2}+\left( c_{1}%
\func{Re}A+c_{3}B\right) h^{2}g_{6} 
\]%
for some constants $c_{1},c_{2},c_{3}$. We know however that 
\begin{equation}
\func{Tr}s=3B+4\func{Re}A  \label{trs1}
\end{equation}%
Hence, 
\begin{eqnarray}
s &=&c_{2}BG^{2}dr^{2}+\left( c_{1}\func{Re}A+c_{3}B\right) h^{2}g_{6}-\frac{%
1}{2}\left( 3B+4\func{Re}A\right) g_{7}  \nonumber \\
&=&\left( \left( c_{2}-\frac{3}{2}\right) B-2\func{Re}A\right)
G^{2}dr^{2}+\left( \left( c_{1}-2\right) \func{Re}A+\left( c_{3}-\frac{3}{2}%
\right) B\right) h^{2}g_{6}  \label{chis1}
\end{eqnarray}%
However we also have 
\[
\chi _{abc}=s_{[a}^{\ d}\varphi _{\left\vert d\right\vert bc]} 
\]%
Note that the $\Omega $ and $\bar{\Omega}$ terms in $\chi $ are obtained
from contraction of the $g_{6}$ term in $s$ with the $\Omega $ and $\bar{%
\Omega}$ terms in $\varphi $. Therefore, the factor in front of $g_{6}$ in $%
s $ must be independent of $B$. Thus, $c_{3}=\frac{3}{2}$. Using this, we
take the trace of (\ref{chis1}), and obtain 
\begin{eqnarray}
\func{Tr}s &=&\left( \left( c_{2}-\frac{3}{2}\right) B-2\func{Re}A\right)
+6\left( c_{1}-2\right) \func{Re}A  \nonumber \\
&=&\left( c_{2}-\frac{3}{2}\right) B+\left( 6c_{1}-14\right) \func{Re}A
\label{trs2}
\end{eqnarray}
Comparing coefficients of $B$ and $\func{Re}A$ in (\ref{trs1}) and (\ref%
{trs2}) we conclude that 
\begin{eqnarray*}
c_{1} &=&3 \\
c_{2} &=&\frac{9}{2}
\end{eqnarray*}%
Therefore, indeed, 
\[
s=\left( 3B-2\func{Re}A\right) G^{2}dr^{2}+\left( \func{Re}A\right)
h^{2}g_{6} 
\]
\end{proof}

Now given the $3$-form$\chi $ (\ref{gen3form}), work out $d\chi $ and $\ast
d\chi .$

\begin{proposition}
\label{Propdchi}Suppose $\chi $ is an $SU\left( 3\right) $-equivariant $3$%
-form on $M^{7}$ given by (\ref{gen3form}). Then using the notation in (\ref%
{chinotation}), the components of $\ast d\chi $ are 
\begin{subequations}%
\label{stdchicpts}%
\begin{eqnarray}
\func{Re}_{1}\left( \ast d\chi \right) &=&G^{-1}\left( \func{Im}A^{\prime
}+3h^{-1}h^{\prime }\func{Im}A+\theta ^{\prime }\func{Re}A-3\lambda
BGh^{-1}\sin \theta \right) \\
\func{Im}_{1}\left( \ast d\chi \right) &=&G^{-1}\left( -\func{Re}A^{\prime
}-3h^{-1}h^{\prime }\func{Re}A+\theta ^{\prime }\func{Im}A-3\lambda
BGh^{-1}\cos \theta \right) \\
\func{Re}_{2}\left( \ast d\chi \right) &=&-4\lambda h^{-1}\left( \sin \theta 
\func{Re}A+\cos \theta \func{Im}A\right)
\end{eqnarray}%
\end{subequations}
where $^{\prime }$ denotes differentiation with respect to $r$.
\end{proposition}

\begin{proof}
Using the $SU\left( 3\right) $-structure properties (\ref{su3struct}), we
have 
\begin{eqnarray*}
d\chi &=&\frac{1}{2}\left( AF^{3}\right) ^{\prime }dr\wedge \Omega +\frac{1}{%
2}\left( \bar{A}\bar{F}^{3}\right) ^{\prime }dr\wedge \bar{\Omega}+\frac{1}{2%
}AF^{3}d\Omega +\frac{1}{2}\bar{A}\bar{F}^{3}d\bar{\Omega}-BGh^{2}dr\wedge
d\omega \\
&=&\frac{1}{2}\left( A^{\prime }F^{3}+A\left( F^{3}\right) ^{\prime
}+3\lambda BGh^{2}\right) dr\wedge \Omega +\frac{1}{2}\left( \bar{A}^{\prime
}\bar{F}^{3}+\bar{A}\left( \bar{F}^{3}\right) ^{\prime }+\frac{3}{2}\lambda
BGh^{2}\right) dr\wedge \bar{\Omega} \\
&&+\lambda i\left( AF^{3}-\bar{A}\bar{F}^{3}\right) \omega ^{2}
\end{eqnarray*}%
Taking the Hodge star, and using (\ref{hodgeduals}) we obtain 
\begin{eqnarray*}
\ast d\chi &=&-\frac{1}{2}iG^{-1}\left( A^{\prime }F^{3}+A\left(
F^{3}\right) ^{\prime }+3\lambda BGh^{2}\right) \Omega +\frac{1}{2}%
iG^{-1}\left( \bar{A}^{\prime }\bar{F}^{3}+\bar{A}\left( \bar{F}^{3}\right)
^{\prime }+\frac{3}{2}\lambda BGh^{2}\right) \bar{\Omega} \\
&&+2\lambda ih^{-2}G\left( AF^{3}-\bar{A}\bar{F}^{3}\right) dr\wedge \omega
\\
&=&-\frac{1}{2}iG^{-1}\left( A^{\prime }+A\frac{\left( F^{3}\right) ^{\prime
}\bar{F}^{3}}{h^{6}}+3\lambda BGh^{-4}\bar{F}^{3}\right) F^{3}\Omega \\
&&+\frac{1}{2}iG^{-1}\left( \bar{A}^{\prime }+\bar{A}\frac{\left( \bar{F}%
^{3}\right) ^{\prime }F^{3}}{h^{6}}+3\lambda BGh^{-4}F^{3}\right) \bar{F}^{3}%
\bar{\Omega} \\
&&+2\lambda ih^{-2}G\left( AF^{3}-\bar{A}\bar{F}^{3}\right) dr\wedge \omega
\end{eqnarray*}%
Note that 
\[
\frac{\left( F^{3}\right) ^{\prime }\bar{F}^{3}}{h^{6}}=3h^{-1}h^{\prime
}+i\theta ^{\prime } 
\]%
So, 
\begin{eqnarray*}
\ast d\chi &=&\frac{1}{2}G^{-1}\left( -iA^{\prime }-3iAh^{-1}h^{\prime
}+A\theta ^{\prime }-3i\lambda BGh^{-4}\bar{F}^{3}\right) F^{3}\Omega \\
&&+\frac{1}{2}G^{-1}\left( i\bar{A}^{\prime }+3i\bar{A}h^{-1}h^{\prime }+%
\bar{A}\theta ^{\prime }+3i\lambda BGh^{-4}F^{3}\right) \bar{F}^{3}\bar{%
\Omega} \\
&&+2\lambda ih^{-2}G\left( AF^{3}-\bar{A}\bar{F}^{3}\right) dr\wedge \omega
\end{eqnarray*}%
Thus, 
\begin{eqnarray*}
\func{Re}_{1}\left( \ast d\chi \right) &=&G^{-1}\left( \func{Im}A^{\prime
}+3h^{-1}h^{\prime }\func{Im}A+\theta ^{\prime }\func{Re}A-3\lambda
BGh^{-1}\sin \theta \right) \\
\func{Im}_{1}\left( \ast d\chi \right) &=&G^{-1}\left( -\func{Re}A^{\prime
}-3h^{-1}h^{\prime }\func{Re}A+\theta ^{\prime }\func{Im}A-3\lambda
BGh^{-1}\cos \theta \right) \\
\func{Re}_{2}\left( \ast d\chi \right) &=&-4\lambda h^{-4}\func{Im}\left(
AF^{3}\right) \\
&=&-4\lambda h^{-1}\left( \sin \theta \func{Re}A+\cos \theta \func{Im}%
A\right) .
\end{eqnarray*}
\end{proof}

Similarly, we can work out $d\ast \chi $ and $\ast d\ast \chi $.

\begin{proposition}
\label{Propdschi}Suppose $\chi $ is an $SU\left( 3\right) $-equivariant $3$%
-form on $M^{7}$ given by (\ref{gen3form}). Then,%
\begin{equation}
\ast \chi =\frac{i}{2}AF^{3}Gdr\wedge \Omega -\frac{i}{2}\bar{A}\bar{F}%
^{3}Gdr\wedge \bar{\Omega}+\frac{1}{2}h^{4}B\omega ^{2}  \label{schi}
\end{equation}%
and 
\begin{eqnarray}
d\ast \chi &=&\left( \frac{1}{2}\frac{B^{\prime }h}{G}+\frac{2h^{\prime }B}{G%
}+2\lambda \left( \cos \theta \func{Re}A-\sin \theta \func{Im}A\right)
\right) Gh^{3}dr\wedge \omega ^{2}  \label{dschi} \\
\ast d\ast \chi &=&4h^{-1}\left( \frac{1}{4}\frac{B^{\prime }h}{G}+\frac{%
h^{\prime }B}{G}+\lambda \left( \cos \theta \func{Re}A-\sin \theta \func{Im}%
A\right) \right) \left( G^{-1}\frac{\partial }{\partial r}\right) \lrcorner
\varphi  \label{sdschi}
\end{eqnarray}%
In particular, $d\ast \chi \in \Lambda _{7}^{5}$ and $\ast d\ast \chi \in
\Lambda _{7}^{2}.$
\end{proposition}

\begin{proof}
To find $\ast \chi $ we just apply (\ref{hodgeduals}) to $\chi $ (\ref%
{gen3form}): 
\begin{eqnarray*}
\ast \chi &=&\ast \left( \frac{1}{2}AF^{3}\Omega +\frac{1}{2}\bar{A}\bar{F}%
^{3}\bar{\Omega}+Gh^{2}Bdr\wedge \omega \right) \\
&=&\frac{i}{2}AF^{3}Gdr\wedge \Omega -\frac{i}{2}\bar{A}\bar{F}^{3}Gdr\wedge 
\bar{\Omega}+\frac{1}{2}h^{4}B\omega ^{2}
\end{eqnarray*}%
Then, to differentiate this, we use (\ref{su3struct}): 
\begin{eqnarray*}
d\left( \ast \chi \right) &=&-\frac{i}{2}AF^{3}Gdr\wedge d\Omega +\frac{i}{2}%
\bar{A}\bar{F}^{3}Gdr\wedge d\bar{\Omega}+\frac{1}{2}\left( h^{4}B\right)
^{\prime }dr\wedge \omega ^{2}+\frac{1}{2}h^{4}B\left( d\omega ^{2}\right) \\
&=&\lambda AF^{3}Gdr\wedge \omega ^{2}+\lambda \bar{A}\bar{F}^{3}Gdr\wedge
\omega ^{2}+\frac{1}{2}\left( h^{4}B\right) ^{\prime }dr\wedge \omega ^{2} \\
&=&\left( \lambda \left( AF^{3}+\bar{A}\bar{F}^{3}\right) +\frac{1}{2}\frac{%
\left( h^{4}B\right) ^{\prime }}{G}\right) Gdr\wedge \omega ^{2} \\
&=&\left( \frac{1}{2}\frac{\left( h^{4}B\right) ^{\prime }}{G}+2\lambda 
\func{Re}\left( AF^{3}\right) \right) Gdr\wedge \omega ^{2} \\
&=&\left( \frac{1}{2}\frac{B^{\prime }h}{G}+\frac{2h^{\prime }B}{G}+2\lambda
\left( \cos \theta \func{Re}A-\sin \theta \func{Im}A\right) \right)
Gh^{3}dr\wedge \omega ^{2}
\end{eqnarray*}%
Applying (\ref{hodgeduals}) again, and using the expression for $\varphi $ (%
\ref{g2phi}), we get (\ref{sdschi}).
\end{proof}

To work out the torsion of the $G_{2}$-structure $\left( \varphi ,g\right) $
on $M^{7}$ we can use Propositions \ref{Propdchi} and \ref{Propdschi} in a
very important special case when $\chi =\varphi .$

\begin{corollary}
\label{Corrdphidpsi}In the notation of (\ref{chinotation}), the components
of $\ast d\varphi $ are given by 
\begin{subequations}%
\label{sdphicpts} 
\begin{eqnarray}
\func{Re}_{1}\left( \ast d\varphi \right) &=&\frac{\theta ^{\prime }}{G}%
-3\lambda h^{-1}\sin \theta \\
\func{Im}_{1}\left( \ast d\varphi \right) &=&-3G^{-1}h^{-1}\left( h^{\prime
}+\lambda G\cos \theta \right) \\
\func{Re}_{2}\left( \ast d\varphi \right) &=&-4\lambda h^{-1}\sin \theta
\end{eqnarray}%
\end{subequations}
and 
\begin{equation}
\ast d\psi =4h^{-1}\left( \frac{h^{\prime }}{G}+\lambda \cos \theta \right)
\left( G^{-1}\frac{\partial }{\partial r}\right) \lrcorner \varphi
\label{sdpsi}
\end{equation}
\end{corollary}

\begin{proof}
We set $\func{Re}A=1,$ $\func{Im}A=0$ and $B=1$ in Propositions \ref%
{Propdchi} and \ref{Propdschi} and thus obtain (\ref{sdphicpts}) and (\ref%
{sdpsi}).
\end{proof}

Combining Proposition \ref{PropChiG2cpts} and Corollary \ref{Corrdphidpsi},
we obtain the torsion components of the $G_{2}$-structure $\left( \varphi
,g\right) .$

\begin{theorem}
\label{ThmTorsionCpts}The torsion components of the warped product $G_{2}$%
-structure $\left( \varphi ,g\right) $ on $M^{7}$ are given by 
\begin{subequations}%
\label{torsioncpts} 
\begin{eqnarray}
\tau _{1} &=&\frac{1}{7}\left( \frac{\theta ^{\prime }}{G}-\frac{6\lambda
\sin \theta }{h}\right)  \label{tau1} \\
\tau _{7} &=&-h^{-1}\left( \frac{h^{\prime }}{G}+\lambda \cos \theta \right)
Gdr  \label{tau7} \\
\tau _{14} &=&0  \label{tau14} \\
\tau _{27}^{\sharp } &=&\frac{1}{7}\left( \frac{\theta ^{\prime }}{G}+\frac{%
\lambda \sin \theta }{h}\right) \func{diag}\left( 6,-\delta _{6}\right)
\label{tau27}
\end{eqnarray}%
\end{subequations}%
where $\tau _{27}^{\sharp }$ denotes $\tau _{27}$ with one raised index.
Correspondingly, the full torsion tensor $T$ is given by 
\begin{equation}
T^{\sharp }=\func{diag}\left( \frac{\theta ^{\prime }}{G},-\frac{\lambda
\sin \theta }{h}\delta _{6}\right) -h^{-1}\left( \frac{h^{\prime }}{G}%
+\lambda \cos \theta \right) J_{6}  \label{torsiontensor}
\end{equation}%
where $J_{6}$ is the (almost) complex structure on $M^{6}.$
\end{theorem}

\begin{remark}
Expressions for torsion components of this warped product $G_{2}$-structure
have originally been derived by Cleyton and Ivanov in \cite%
{CleytonIvanovCurv} and also later on by Karigiannis, McKay and Tsui in \cite%
{KarigiannisMcKayTsui}. However here we give the $\tau _{27}$ component as a 
$2$-tensor rather a $3$-form, and we also give the expression for the full
torsion tensor. To the author's knowledge these formulae have not appeared
in the literature.
\end{remark}

\begin{proof}[Proof of Theorem \protect\ref{ThmTorsionCpts}]
From Proposition \ref{PropChiG2cpts}, we know that if we write 
\[
\ast d\varphi =X\lrcorner \psi +\mathrm{i}_{\varphi }\left( s\right) 
\]%
then,%
\begin{eqnarray*}
X &=&\left( \func{Im}_{1}\left( \ast d\varphi \right) \right) G^{-1}\frac{%
\partial }{\partial r} \\
s &=&\left( 3\func{Re}_{2}\left( \ast d\varphi \right) -2\func{Re}_{1}\left(
\ast d\varphi \right) \right) G^{2}dr^{2}+\func{Re}_{1}\left( \ast d\varphi
\right) h^{2}g_{6} \\
\func{Tr}s &=&3\func{Re}_{2}\left( \ast d\varphi \right) +4\func{Re}%
_{1}\left( \ast d\varphi \right)
\end{eqnarray*}%
Using Corollary \ref{Corrdphidpsi}, we thus obtain 
\begin{eqnarray*}
X &=&-3h^{-1}\left( \frac{h^{\prime }}{G}+\lambda \cos \theta \right) G^{-1}%
\frac{\partial }{\partial r} \\
s &=&-2\left( \frac{\theta ^{\prime }}{G}+3\lambda h^{-1}\sin \theta \right)
G^{2}dr^{2}+\left( \frac{\theta ^{\prime }}{G}-3\lambda h^{-1}\sin \theta
\right) h^{2}g_{6} \\
\func{Tr}s &=&4\left( \frac{\theta ^{\prime }}{G}-6\lambda h^{-1}\sin \theta
\right)
\end{eqnarray*}%
Recall that 
\[
d\varphi =4\tau _{1}\psi -3\tau _{7}\wedge \varphi -3\ast \mathrm{i}%
_{\varphi }\left( \tau _{27}\right) 
\]%
and hence%
\[
\ast d\varphi =4\tau _{1}\varphi +3\tau _{7}^{\sharp }\lrcorner \psi -3%
\mathrm{i}_{\varphi }\left( \tau _{27}\right) 
\]%
Now, 
\begin{eqnarray*}
3\tau _{7}^{\sharp } &=&X \\
4\tau _{1} &=&\frac{1}{7}\func{Tr}s \\
-3\tau _{27} &=&s-\frac{1}{7}\left( \func{Tr}s\right) g_{7}
\end{eqnarray*}%
Hence immediately obtain expressions for $\tau _{1}$ and $\tau _{7}$. From
this, we also get $\tau _{27}$ 
\begin{eqnarray*}
\tau _{27} &=&-\frac{1}{3}s+\frac{1}{21}\left( \func{Tr}s\right) g_{7} \\
&=&\frac{6}{7}\left( \frac{\theta ^{\prime }}{G}+\lambda h^{-1}\sin \theta
\right) G^{2}dr^{2}-\frac{1}{7}\left( \frac{\theta ^{\prime }}{G}+\lambda
h^{-1}\sin \theta \right) h^{2}g_{6} \\
&=&\frac{1}{7}\left( \frac{\theta ^{\prime }}{G}+\lambda h^{-1}\sin \theta
\right) \left( 6G^{2}dr^{2}-h^{2}g_{6}\right)
\end{eqnarray*}%
Raising one index on $\tau _{27}$ we obtain (\ref{tau27}).

Recall that%
\[
d\psi =-4\tau _{7}\wedge \psi -2\ast \tau _{14} 
\]%
and hence 
\[
\ast d\psi =-4\tau _{7}^{\sharp }\lrcorner \varphi -2\tau _{14}. 
\]%
However, from Proposition \ref{Propdschi} we conclude that $\pi _{14}\left(
\ast d\psi \right) =0$, and thus $\tau _{14}=0.$

To obtain the full torsion tensor, we just calculate 
\begin{eqnarray*}
T &=&\tau _{1}g+\tau _{7}\lrcorner \varphi +\tau _{14}+\tau _{27} \\
&=&\left( \frac{\theta ^{\prime }}{G}\right) G^{2}dr^{2}-\left( \lambda
h^{-1}\sin \theta \right) h^{2}g_{6}-h^{-1}\left( \frac{h^{\prime }}{G}%
+\lambda \cos \theta \right) \omega .
\end{eqnarray*}%
Raising the first index using $g_{7}^{-1}$, we get (\ref{torsiontensor}).
\end{proof}

\begin{example}
Suppose $M^{6}$ is a Calabi-Yau manifold, then the torsion tensor is given
by 
\[
T^{\sharp }=\func{diag}\left( \frac{\theta ^{\prime }}{G},0\right)
-h^{-1}\left( \frac{h^{\prime }}{G}\right) J_{6} 
\]%
The $G_{2}$-structure is then torsion-free if and only if $\theta ^{\prime }$
and $h^{\prime }$ both vanish. After redefining the $r$ coordinate to set $%
G=1$, the $G_{2}$-structure is then given by 
\[
\varphi =\func{Re}\left( h^{3}e^{i\theta }\Omega \right) +dr\wedge \left(
h^{2}\omega \right) . 
\]%
This is just a direct product $G_{2}$-structure which is obtained from an $%
SU\left( 3\right) $-structure which is obtained from the original one by a
constant phase factor on $\Omega $ and an overall constant conformal factor $%
h$.
\end{example}

Note that if $M^{6}$ is nearly K\"{a}hler, so that $\lambda \neq 0$, then in
order to have $T=0$, we still need $\theta ^{\prime }=0$. Moreover, we also
must have $\sin \theta =0.$ Thus, $\theta =k\pi $ for some integer $\pi .$
This sets both $\tau _{1}$ and $\tau _{27}$ components to zero. In order to
have $\tau _{7}=0$, we then also need $G^{-1}h^{\prime }+\lambda \cos \theta
=0.$ Since $\theta =k\pi $, $\cos \theta =\pm 1.$ So, must have $h^{\prime
}=\pm \lambda G.$

For convenience, let 
\begin{subequations}%
\label{abc} 
\begin{eqnarray}
\alpha &=&\frac{\theta ^{\prime }}{G}  \label{alpha} \\
\beta &=&\lambda h^{-1}\sin \theta  \label{beta} \\
\gamma &=&\frac{h^{\prime }}{h}+\lambda h^{-1}G\cos \theta  \label{gamma}
\end{eqnarray}%
\end{subequations}%
. Then in terms of $\alpha ,\beta ,\gamma $, the non-vanishing torsion
components are 
\begin{subequations}%
\label{torscompabc} 
\begin{eqnarray}
\tau _{1} &=&\frac{1}{7}\left( \alpha -6\beta \right) \\
\tau _{7} &=&-\gamma dr \\
\tau _{27}^{\sharp } &=&\frac{1}{7}\left( \alpha +\beta \right) \func{diag}%
\left( 6,-\delta _{6}\right)
\end{eqnarray}%
\end{subequations}%
, and the full torsion tensor is then 
\begin{equation}
T^{\sharp }=\func{diag}\left( \alpha ,-\beta \delta _{6}\right)
-G^{-1}\gamma J_{6}  \label{torsabc}
\end{equation}

The components $\alpha ,\beta ,\gamma $ thus uniquely define the torsion
components $\tau _{1},\tau _{7}$ and $\tau _{27}^{\sharp }.$ Also note that
in the important special case of a co-closed $G_{2}$-structure, $\gamma
=0\,. $ In the case when $\lambda =0$ and hence the underlying $6$%
-dimensional space $M^{6}$ is Calabi-Yau, we have $\beta =0$.

Consider what happens to torsion components under a conformal transformation.

\begin{proposition}
\label{PropConfTrans}Under a conformal transformation of the $G_{2}$%
-structure (\ref{g2phi})%
\begin{equation}
\varphi \longrightarrow \tilde{\varphi}=f^{3}\varphi ,  \label{phiconftrans}
\end{equation}%
where $f$ is a nowhere zero function on $L$, the torsion components $\alpha
,\beta ,\gamma $ transform as follows 
\begin{subequations}%
\begin{eqnarray}
\alpha &\longrightarrow &\tilde{\alpha}=f^{-1}\alpha \\
\beta &\longrightarrow &\tilde{\beta}=f^{-1}\beta \\
\gamma &\longrightarrow &\tilde{\gamma}=\frac{f^{\prime }}{f}+\gamma
\end{eqnarray}%
\end{subequations}%
.
\end{proposition}

\begin{proof}
It is well-known \cite{GrigorianG2Torsion1,karigiannis-2005-57} that under
the conformal transformation (\ref{phiconftrans}), the metric $g_{7}$
transforms as 
\[
g_{7}\longrightarrow \tilde{g}_{7}=f^{2}g 
\]%
Note that from (\ref{G2warpedstruct2}) this implies that $\theta $ is
unaffected by the transformation, while 
\begin{subequations}%
\label{Ghconf} 
\begin{eqnarray}
G &\longrightarrow &\tilde{G}=fG \\
h &\longrightarrow &\tilde{h}=fh
\end{eqnarray}%
\end{subequations}%
Thus, from (\ref{abc}) we immediately obtain 
\begin{eqnarray*}
\tilde{\alpha} &=&\frac{\theta ^{\prime }}{\tilde{G}}=f^{-1}\alpha \\
\tilde{\beta} &=&\lambda \tilde{h}^{-1}\sin \theta =f^{-1}\beta \\
\tilde{\gamma} &=&\frac{\tilde{h}^{\prime }}{\tilde{h}}+\lambda \tilde{h}%
^{-1}\tilde{G}\cos \theta \\
&=&\frac{f^{\prime }h+fh^{\prime }}{fh}+\lambda hG\cos \theta =\frac{%
f^{\prime }}{f}+\gamma
\end{eqnarray*}
\end{proof}

Proposition \ref{PropConfTrans} implies that using a suitable conformal
transformation, we can always set $\gamma $, and hence, $\tau _{7}$, to zero.

\begin{corollary}
\label{CorrConf}In (\ref{phiconftrans}), let 
\begin{equation}
f\left( r\right) =e^{-\int_{0}^{r}\gamma \left( s\right) ds}.
\label{ft7zero}
\end{equation}%
Then the transformed $G_{2}$-structure $\tilde{\varphi}$ has $\tilde{\gamma}%
=0$ and hence the $7$-dimensional torsion component $\tilde{\tau}_{7}$ also
vanishes.
\end{corollary}

\begin{remark}
In general, we can always remove the $7$-dimensional component of the
torsion by a conformal transformation if the torsion is in the class $%
\mathbf{1}\oplus \mathbf{7}$ \cite{CleytonIvanovConf, GrigorianG2Torsion1}. $%
G_{2}$-structures in this torsion class are then called \emph{conformally
nearly parallel} $G_{2}$-structures. Corollary \ref{CorrConf} shows that the
warped $G_{2}$-structures (\ref{G2warpedstruct2}) lie in a special subset of 
$G_{2}$-structures of the class $\mathbf{1}\oplus \mathbf{7}\oplus \mathbf{27%
}$ - namely, they are \emph{conformally co-closed }$G_{2}$-structures, since
every such $G_{2}$-structure is conformally equivalent to a co-closed $G_{2}$%
-structure.
\end{remark}

We can use this to show that $\alpha ,\beta ,\gamma $ actually uniquely
determine the $G_{2}$-structure.

\begin{theorem}
Suppose $\alpha ,\beta ,\gamma $, with $\alpha $ and $\beta $ non-zero, and $%
\alpha +\beta $ nowhere zero, are torsion components of some $G_{2}$%
-structure on $M^{7}$ with $\lambda \neq 0$. Then the functions $\theta ,h,G$
are uniquely defined.
\end{theorem}

\begin{proof}
Suppose we are given $\alpha ,\beta ,\gamma $. We need to show that there
exists a unique solution $\{\theta ,h,G\}$ to equations (\ref{abc}). By
Corollary \ref{CorrConf} we can apply a conformal transformation with $f$
given by (\ref{ft7zero}) to set the $7$-dimensional torsion component to
zero. Equations (\ref{abc}) can then be solved for $\theta ,h,G$ using $%
\tilde{\alpha},\tilde{\beta}$ and $\tilde{\gamma}=0\,.$The original $G$ and $%
h$ can be recovered using (\ref{Ghconf}). Hence without loss of generality
can assume that $\gamma =0$. Therefore, we have equations%
\begin{eqnarray*}
\alpha &=&\frac{\theta ^{\prime }}{G} \\
\beta &=&\lambda h^{-1}\sin \theta \\
h^{\prime } &=&-\lambda G\cos \theta
\end{eqnarray*}%
Consider%
\begin{eqnarray}
\beta ^{\prime } &=&-\lambda h^{-2}h^{\prime }\sin \theta +\lambda
h^{-1}\theta ^{\prime }\cos \theta  \nonumber \\
&=&-\frac{h^{\prime }}{h}\beta -\frac{h^{\prime }}{h}\alpha  \nonumber \\
&=&-\frac{h^{\prime }}{h}\left( \alpha +\beta \right)  \label{bpapb}
\end{eqnarray}%
Hence, 
\begin{equation}
\frac{h^{\prime }}{h}=-\frac{\beta ^{\prime }}{\alpha +\beta }.  \label{hpdh}
\end{equation}%
From this, we get $h$ up to a constant factor $h_{0}\neq 0$:%
\begin{equation}
h=h_{0}e^{-\int_{0}^{r}\frac{\beta ^{\prime }\left( s\right) }{\alpha \left(
s\right) +\beta \left( s\right) }ds}  \label{hh0sol}
\end{equation}%
Furthermore, 
\begin{eqnarray*}
\lambda \sin \theta &=&h\beta \\
\lambda \cos \theta &=&-G^{-1}h^{\prime }
\end{eqnarray*}%
Hence, 
\[
\lambda ^{2}=h^{2}\beta ^{2}+G^{-2}\left( h^{\prime }\right) ^{2} 
\]%
Therefore, 
\[
G^{2}=\frac{\left( h^{\prime }\right) ^{2}}{\lambda ^{2}-h^{2}\beta ^{2}} 
\]%
We also have 
\begin{eqnarray*}
\cot \theta &=&-G^{-1}\beta ^{-1}\frac{h^{\prime }}{h} \\
&=&G^{-1}\frac{\beta ^{\prime }}{\beta }\frac{1}{\alpha +\beta }
\end{eqnarray*}%
where we have used (\ref{hpdh}). Thus, 
\begin{eqnarray*}
\left( \cot \theta \right) \theta ^{\prime } &=&G^{-1}\theta ^{\prime }\frac{%
\beta ^{\prime }}{\beta }\frac{1}{\alpha +\beta } \\
&=&\frac{\beta ^{\prime }}{\beta }\frac{\alpha }{\alpha +\beta }
\end{eqnarray*}%
Integrating, we obtain 
\begin{equation}
\sin \theta =s_{0}e^{\int_{0}^{r}\frac{\beta ^{\prime }\left( s\right) }{%
\beta \left( s\right) }\frac{\alpha \left( s\right) }{\alpha \left( s\right)
+\beta \left( s\right) }ds}  \label{sinthetasol}
\end{equation}%
for some constant $s_{0}$. Using both (\ref{hh0sol}) and (\ref{sinthetasol})
note that 
\begin{eqnarray*}
\beta &=&\frac{\lambda \sin \theta }{h}=\lambda \frac{s_{0}}{h_{0}}%
e^{\int_{0}^{r}\frac{\beta ^{\prime }\left( s\right) }{\beta \left( s\right) 
}\frac{\alpha \left( s\right) }{\alpha \left( s\right) +\beta \left(
s\right) }ds}e^{\int_{0}^{r}\frac{\beta ^{\prime }\left( s\right) }{\alpha
\left( s\right) +\beta \left( s\right) }ds} \\
&=&\lambda \frac{s_{0}}{h_{0}}e^{\int_{0}^{r}\frac{\beta ^{\prime }}{\beta }%
ds} \\
&=&\lambda \frac{s_{0}}{h_{0}}\left\vert \frac{\beta }{\beta \left( 0\right) 
}\right\vert
\end{eqnarray*}%
Note that from (\ref{hh0sol}) $h$ is never zero, so is always either
positive or negative. Similarly, from (\ref{sinthetasol}), $\sin \theta $ is
either always zero (if $s_{0}=0$) or always negative or always positive.
This shows that for consistency $\beta $ is also either always zero, or
always positive or always negative. Therefore, $\left\vert \frac{\beta
\left( r\right) }{\beta \left( 0\right) }\right\vert =\frac{\beta \left(
r\right) }{\beta \left( 0\right) }.$ Thus, 
\[
s_{0}=\frac{\beta \left( 0\right) }{\lambda }h_{0} 
\]%
From the definition of $\alpha $, we can also write 
\begin{equation}
\theta =\int_{0}^{r}\alpha \left( s\right) G\left( s\right) ds+\theta _{0}
\label{thetaintG}
\end{equation}%
Since $\sin \theta _{0}=s_{0},$ substituting (\ref{thetaintG})\ into (\ref%
{sinthetasol}) will fix $s_{0}$.
\end{proof}

\begin{remark}
If $\beta $ is zero (but $\lambda \neq 0$), then $\theta $ must be a
constant integer multiple of $\pi $, and hence $\alpha $ must also be zero.
In this case, $G$ is arbitrary, and $h$ is defined from $G$ up to a constant
multiple. If $\alpha =0$ but $\beta \neq 0$, then we can see that $\theta $
is an arbitrary constant, but $h$ and $G$ are defined as 
\begin{eqnarray*}
h &=&\frac{\lambda \sin \theta }{\beta } \\
G &=&-\frac{h^{\prime }}{\lambda \cos \theta }
\end{eqnarray*}%
whenever $\cos \theta \neq 0$. If however, $\cos \theta =0,$ then $G$ is
arbitrary. Also, suppose $\alpha +\beta =0$, and $\gamma =0.$ Then we have 
\begin{eqnarray*}
\frac{\theta ^{\prime }}{G} &=&-\lambda h^{-1}\sin \theta \\
h^{\prime } &=&-\lambda G\cos \theta
\end{eqnarray*}%
Now, $h^{-1}=-\frac{G^{-1}\theta ^{\prime }}{\lambda \sin \theta }.$ Hence, 
\[
\frac{h^{\prime }}{h}=\frac{\cos \theta }{\sin \theta }\theta ^{\prime } 
\]%
Integrating, we find%
\[
h=A\sin \theta 
\]%
for some constant $A$. But, $h=-\frac{\lambda }{\alpha }\sin \theta ,$ so we
must have $\alpha =-\frac{\lambda }{A},$ which is a constant. Thus $\beta $
is also constant. We then find that $G$ is an arbitrary function, and $%
\theta $ and $h$ are given by 
\begin{eqnarray*}
\theta ^{\prime } &=&\alpha G \\
h &=&-\frac{\lambda }{\alpha }\sin \theta
\end{eqnarray*}
\end{remark}

\section{The Laplacian of $\protect\varphi $}

\setcounter{equation}{0}\label{seclap}Consider the Hodge Laplacian of $%
\varphi $:%
\begin{eqnarray*}
\Delta \varphi &=&dd^{\ast }\varphi +d^{\ast }d\varphi \\
&=&-d\ast d\psi +\ast d\ast d\varphi
\end{eqnarray*}%
Since we know from Theorem \ref{ThmTorsionCpts} that $\tau _{14}=0$ and $%
\tau _{7}=-\gamma dr$, we have 
\begin{eqnarray}
\ast d\psi &=&4\gamma \left( G^{-2}\frac{\partial }{\partial r}\right)
\lrcorner \varphi  \nonumber \\
&=&4\gamma G^{-1}h^{2}\omega  \label{sdpsi2}
\end{eqnarray}%
and moreover, 
\begin{equation}
\ast d\varphi =\frac{1}{2}AF^{3}\Omega +\frac{1}{2}\bar{A}\bar{F}^{3}\bar{%
\Omega}+Gh^{2}Bdr\wedge \omega  \label{sdphi2}
\end{equation}%
where%
\begin{eqnarray*}
\func{Re}A &=&\alpha -3\beta \\
\func{Im}A &=&-3G^{-1}\gamma \\
B &=&-4\beta
\end{eqnarray*}%
Using this we can work out $\Delta \varphi .$

\begin{theorem}
\label{Thmlapcpts}In the notation of (\ref{chinotation}), the components of $%
\Delta \varphi $ are given by 
\begin{subequations}%
\label{lapcpts}%
\begin{eqnarray}
\func{Re}_{1}\left( \Delta \varphi \right) &=&3\gamma G^{-2}\left( -\frac{%
\gamma ^{\prime }}{\gamma }+\frac{\gamma \left( 4\beta -3\alpha \right)
+7\beta ^{\prime }}{\alpha +\beta }\right) +\alpha ^{2}-3\beta \left( \alpha
-4\beta \right) \\
\func{Im}_{2}\left( \Delta \varphi \right) &=&G^{-1}\left( 6\beta ^{\prime
}-\alpha ^{\prime }\right) -6\gamma G^{-1}\alpha \\
\func{Re}_{2}\left( \Delta \varphi \right) &=&-4\beta \left( \alpha -3\beta
\right) +4\gamma G^{-2}\left( -\frac{\gamma ^{\prime }}{\gamma }+\frac{%
\gamma \left( 3\beta -2\alpha \right) +5\beta ^{\prime }}{\alpha +\beta }%
\right)
\end{eqnarray}%
\end{subequations}%
.
\end{theorem}

\begin{proof}
Using the expression (\ref{sdpsi2}) for $\ast d\psi $ together with the
expression (\ref{dlom}) for $d\omega $, we have 
\begin{eqnarray*}
d\ast d\psi &=&4\left( h^{2}G^{-1}\gamma \right) ^{\prime }dr\wedge \omega
+4h^{2}G^{-1}\gamma d\omega \\
&=&-6\lambda h^{2}G^{-1}\gamma \left( \Omega +\bar{\Omega}\right) +4\left(
h^{2}G^{-1}\gamma \right) ^{\prime }dr\wedge \omega \\
&=&\frac{1}{2}\left( -12\lambda h^{-4}G^{-1}\gamma \bar{F}^{3}\right)
F^{3}\Omega +\frac{1}{2}\left( -12\lambda h^{-4}G^{-1}\gamma F^{3}\right) 
\bar{F}^{3}\bar{\Omega} \\
&&+4\left( G^{-1}h^{-2}\left( h^{2}G^{-1}\gamma \right) ^{\prime }\right)
Gh^{2}dr\wedge \omega \\
&=&\frac{1}{2}\left( -12\lambda h^{-4}\gamma G^{-1}\bar{F}^{3}\right)
F^{3}\Omega +\frac{1}{2}\left( -12\lambda h^{-4}\gamma G^{-1}F^{3}\right) 
\bar{F}^{3}\bar{\Omega} \\
&&+4\left( 2G^{-2}\frac{h^{\prime }}{h}\gamma +G^{-2}\gamma ^{\prime
}-G^{-3}G^{\prime }\gamma \right) Gh^{2}dr\wedge \omega
\end{eqnarray*}%
Hence, 
\begin{eqnarray*}
\func{Re}_{1}\left( d\ast d\psi \right) &=&-12\lambda h^{-1}G^{-1}\gamma
\cos \theta \\
&=&-12G^{-2}\gamma \left( \gamma -\frac{h^{\prime }}{h}\right) \\
\func{Im}_{1}\left( d\ast d\psi \right) &=&12\lambda h^{-1}G^{-1}\gamma \sin
\theta \\
&=&12\beta G^{-1}\gamma \\
\func{Re}_{2}\left( d\ast d\psi \right) &=&4G^{-2}\gamma \left( 2\frac{%
h^{\prime }}{h}+\frac{\gamma ^{\prime }}{\gamma }-\frac{G^{\prime }}{G}%
\right)
\end{eqnarray*}%
Similarly, using the expression (\ref{sdphi2}) for $\ast d\varphi $ together
with Proposition \ref{Propdchi}, we can work out $\ast d\ast \varphi $:%
\begin{eqnarray*}
\func{Re}_{1}\left( \ast d\ast d\varphi \right) &=&G^{-1}\left( \func{Im}%
A^{\prime }+3h^{-1}h^{\prime }\func{Im}A+\theta ^{\prime }\func{Re}%
A-3\lambda BGh^{-1}\sin \theta \right) \\
&=&-3G^{-1}\left( G^{-1}\gamma \right) ^{\prime }-9G^{-2}\frac{h^{\prime }}{h%
}\gamma +\alpha ^{2}-3\alpha \beta +12\lambda \beta h^{-1}\sin \theta \\
&=&-3G^{-2}\gamma ^{\prime }+3G^{-3}G^{\prime }\gamma -9G^{-2}\frac{%
h^{\prime }}{h}\gamma +\alpha ^{2}-3\alpha \beta +12\beta ^{2} \\
&=&-\frac{3\gamma }{G^{2}}\left( \frac{\gamma ^{\prime }}{\gamma }-\frac{%
G^{\prime }}{G}+3\frac{h^{\prime }}{h}\right) +\alpha ^{2}-3\alpha \beta
+12\beta ^{2} \\
\func{Im}_{1}\left( \ast d\ast d\varphi \right) &=&G^{-1}\left( -\func{Re}%
A^{\prime }-3h^{-1}h^{\prime }\func{Re}A+\theta ^{\prime }\func{Im}%
A-3\lambda BGh^{-1}\cos \theta \right) \\
&=&-G^{-1}\alpha ^{\prime }+3G^{-1}\beta ^{\prime }-\frac{3h^{\prime }}{Gh}%
\left( \alpha -3\beta \right) -3\alpha G^{-1}\gamma +12\lambda \beta
h^{-1}\cos \theta \\
&=&-G^{-1}\alpha ^{\prime }+3G^{-1}\beta ^{\prime }-\frac{3h^{\prime }\alpha 
}{Gh}+\frac{9\beta h^{\prime }}{Gh}-3\alpha G^{-1}\gamma +12\beta
G^{-1}\gamma -12\beta \frac{h^{\prime }}{Gh} \\
&=&-G^{-1}\alpha ^{\prime }+3G^{-1}\beta ^{\prime }-\frac{3h^{\prime }\alpha 
}{Gh}-\frac{3\beta h^{\prime }}{Gh}-3\alpha G^{-1}\gamma +12\beta
G^{-1}\gamma \\
&=&-G^{-1}\left( \alpha ^{\prime }-3\beta ^{\prime }\right) -\frac{%
3h^{\prime }}{Gh}\left( \alpha +\beta \right) -3G^{-1}\gamma \left( \alpha
-4\beta \right) \\
\func{Re}_{2}\left( \ast d\ast d\varphi \right) &=&-4\lambda h^{-1}\left(
\sin \theta \func{Re}A+\cos \theta \func{Im}A\right) \\
&=&-4\lambda h^{-1}\left( \sin \theta \left( \alpha -3\beta \right)
-3G^{-1}\gamma \cos \theta \right) \\
&=&-4\beta \left( \alpha -3\beta \right) +12G^{-2}\gamma \left( \gamma -%
\frac{h^{\prime }}{h}\right) \\
&=&-4\beta \left( \alpha -3\beta \right) +12G^{-2}\gamma \left( \gamma -%
\frac{h^{\prime }}{h}\right)
\end{eqnarray*}%
Note that similarly to (\ref{hpdh}), we can express $\frac{h^{\prime }}{h}$
in terms of $\alpha ,\beta ,\gamma $. For this, consider $\beta ^{\prime }$:%
\begin{eqnarray*}
\beta ^{\prime } &=&-\frac{h^{\prime }}{h}\frac{\lambda \sin \theta }{h}+%
\frac{\theta ^{\prime }\lambda \cos \theta }{h} \\
&=&-\frac{h^{\prime }}{h}\beta +\alpha \frac{\lambda G\cos \theta }{h} \\
&=&-\frac{h^{\prime }}{h}\beta +\alpha \left( \gamma -\frac{h^{\prime }}{h}%
\right) \\
&=&-\left( \alpha +\beta \right) \frac{h^{\prime }}{h}+\alpha \gamma \\
\frac{h^{\prime }}{h} &=&\frac{\alpha \gamma -\beta ^{\prime }}{\alpha
+\beta }
\end{eqnarray*}
\end{proof}

From this, 
\[
G^{-1}\frac{h^{\prime }}{h}\left( \alpha +\beta \right) =G^{-1}\left( \alpha
\gamma -\beta ^{\prime }\right) 
\]%
and 
\begin{equation}
\gamma -\frac{h^{\prime }}{h}=\frac{\beta \gamma +\beta ^{\prime }}{\alpha
+\beta }  \label{gamhph}
\end{equation}%
Hence, we get 
\begin{eqnarray*}
\func{Re}_{1}\left( \ast d\ast d\varphi \right) &=&-3G^{-2}\gamma \left( 
\frac{\gamma ^{\prime }}{\gamma }+3\frac{\alpha \gamma -\beta ^{\prime }}{%
\alpha +\beta }\right) +\alpha ^{2}-3\alpha \beta +12\beta ^{2} \\
\func{Im}_{1}\left( \ast d\ast d\varphi \right) &=&-G^{-1}\left( \alpha
^{\prime }-3\beta ^{\prime }\right) -3G^{-1}\left( \alpha \gamma -\beta
^{\prime }\right) -3G^{-1}\gamma \left( \alpha -4\beta \right) \\
&=&G^{-1}\left( 6\beta ^{\prime }-\alpha ^{\prime }\right) +6G^{-1}\gamma
\left( 2\beta -\alpha \right) \\
\func{Re}_{2}\left( \ast d\ast d\varphi \right) &=&-4\beta \left( \alpha
-3\beta \right) +12G^{-2}\gamma \left( \frac{\beta \gamma +\beta ^{\prime }}{%
\alpha +\beta }\right)
\end{eqnarray*}%
and 
\begin{eqnarray*}
\func{Re}_{1}\left( d\ast d\psi \right) &=&-12G^{-2}\gamma \left( \frac{%
\beta \gamma +\beta ^{\prime }}{\alpha +\beta }\right) \\
\func{Im}_{1}\left( d\ast d\psi \right) &=&12\beta G^{-1}\gamma \\
\func{Re}_{2}\left( d\ast d\psi \right) &=&4\gamma G^{-2}\left( \frac{\gamma
^{\prime }}{\gamma }+2\frac{\alpha \gamma -\beta ^{\prime }}{\alpha +\beta }%
\right)
\end{eqnarray*}%
Combining, we finally obtain the expressions (\ref{lapcpts}) for the
components of the Laplacian.

By setting $\beta =0$ we obtain the Laplacian in the Calabi-Yau case.

\begin{corollary}
Suppose $\lambda =0$, so that $M^{6}$ is Calabi-Yau. The components of the
Laplacian of $\varphi $ are then given by: 
\begin{eqnarray}
\func{Re}_{1}\left( \Delta \varphi \right) &=&-3\gamma G^{-2}\left( \frac{%
\gamma ^{\prime }}{\gamma }+3\gamma \right) +\alpha ^{2} \\
\func{Im}_{2}\left( \Delta \varphi \right) &=&-\alpha G^{-1}\left( \frac{%
\alpha ^{\prime }}{\alpha }+6\gamma \right) \\
\func{Re}_{2}\left( \Delta \varphi \right) &=&-4\gamma G^{-2}\left( \frac{%
\gamma ^{\prime }}{\gamma }+2\gamma \right)
\end{eqnarray}%
where $\alpha =G^{-1}\theta ^{\prime }$ and $\gamma =h^{-1}h^{\prime }$.
\end{corollary}

In the case when the $G_{2}$-structure is co-closed, we get the components
of the Laplacian by setting $\gamma =0.$

\begin{corollary}
\label{CorrLapgam0}Suppose the $G_{2}$-structure $\varphi $ is co-closed, so
that $\gamma =0$. Then the components of the Laplacian of $\varphi $ are
given by 
\begin{subequations}%
\label{lapcptsgam0}%
\begin{eqnarray}
\func{Re}_{1}\left( \Delta \varphi \right) &=&\alpha ^{2}-3\beta \alpha
+12\beta ^{2}  \label{re1lap} \\
\func{Im}_{2}\left( \Delta \varphi \right) &=&G^{-1}\left( 6\beta ^{\prime
}-\alpha ^{\prime }\right)  \label{im1lap} \\
\func{Re}_{2}\left( \Delta \varphi \right) &=&-4\beta \left( \alpha -3\beta
\right)  \label{re2lap}
\end{eqnarray}%
\end{subequations}
Moreover, if $M^{6}$ is Calabi-Yau, then 
\begin{subequations}%
\begin{eqnarray}
\func{Re}_{1}\left( \Delta \varphi \right) &=&\alpha ^{2} \\
\func{Im}_{2}\left( \Delta \varphi \right) &=&-G^{-1}\alpha ^{\prime } \\
\func{Re}_{2}\left( \Delta \varphi \right) &=&0
\end{eqnarray}%
\end{subequations}%
\end{corollary}

\begin{remark}
Expressions for the Laplacian of a warped product $G_{2}$-structure have
been first computed by Karigiannis, McKay and Tsui in \cite%
{KarigiannisMcKayTsui}. The expressions in \cite{KarigiannisMcKayTsui} were
given for both the Calabi-Yau and nearly K\"{a}hler cases, but only when the
co-closed condition was imposed.
\end{remark}

It is a well-known consequence of Hodge's Theorem that on a compact manifold
a harmonic form is both closed and co-closed. For non-compact manifolds this
may not be true in general. However from Corollary \ref{CorrLapgam0}, it is
easy to see that if $\varphi $ is co-closed and harmonic, then $\alpha
=\beta =0,$ and hence it is torsion-free, and thus also closed.

\section{Flows of warped $G_{2}$-structures}

\setcounter{equation}{0}\label{secflow}Suppose now $\varphi \left( t\right) $
is a family of $G_{2}$-structures on $M^{7}$ defined for $t\in \lbrack 0,T)$%
, such that for every $t,$ $\varphi \left( t\right) $ is of the form (\ref%
{g2phi}). In particular, we will assume that the underlying $SU\left(
3\right) $ structure is constant and only the parameters $G,h,\theta $ of
the warped product depend on $t.$ Since we will be interested in flows of
the dual form $\psi $, we need to know how the evolution of $\psi \left(
t\right) $, as well as the evolution of the quantities $\alpha ,\beta
,\gamma $, is related to the time evolution of $G,h,\theta $.

\begin{lemma}
\label{LemPsidot}Suppose $G\left( t\right) ,h\left( t\right) $ and $\theta
\left( t\right) $ define a time-dependent family of $G_{2}$-structures $%
\varphi \left( t\right) $ via (\ref{g2phi}). Then, for $\psi \left( t\right)
=\ast _{t}\varphi \left( t\right) $, we have%
\begin{subequations}%
\label{psidot}%
\begin{eqnarray}
\func{Re}_{1}\left( \ast _{t}\frac{\partial }{\partial t}\psi \right)
&=&G^{-1}\dot{G}+3h^{-1}\dot{h} \\
\func{Im}_{1}\left( \ast _{t}\frac{\partial }{\partial t}\psi \right) &=&%
\dot{\theta} \\
\func{Re}_{2}\left( \ast _{t}\frac{\partial }{\partial t}\psi \right)
&=&4h^{-1}\dot{h}
\end{eqnarray}%
\end{subequations}
where the dot denotes time derivative. Also, 
\begin{subequations}%
\label{abcdot} 
\begin{eqnarray}
\dot{\alpha} &=&\frac{\left( \dot{\theta}\right) ^{\prime }}{G}-\alpha \frac{%
\dot{G}}{G} \\
\dot{\beta} &=&\dot{\theta}G^{-1}\left( \frac{\beta \gamma +\beta ^{\prime }%
}{\alpha +\beta }\right) -\frac{\dot{h}}{h}\beta \\
\dot{\gamma} &=&\left( \frac{\dot{h}}{h}\right) ^{\prime }+\left( \frac{\dot{%
G}}{G}-\frac{\dot{h}}{h}\right) \left( \frac{\beta \gamma +\beta ^{\prime }}{%
\alpha +\beta }\right) -\dot{\theta}G\beta
\end{eqnarray}%
\end{subequations}%
\end{lemma}

\begin{proof}
Consider $\dot{\psi}:$%
\begin{eqnarray*}
\frac{\partial }{\partial t}\psi &=&\frac{\partial }{\partial t}\left( \frac{%
1}{2}h^{4}\omega ^{2}-\frac{iGF^{3}}{2}\Omega \wedge dr+\frac{iG\bar{F}^{3}}{%
2}\bar{\Omega}\wedge dr\right) \\
&=&2h^{3}\dot{h}\omega ^{2}-\frac{i}{2}\frac{\partial }{\partial t}\left(
GF^{3}\right) \Omega \wedge dr+\frac{i}{2}\frac{\partial }{\partial t}\left(
G\bar{F}^{3}\right) \bar{\Omega}\wedge dr \\
&=&2h^{3}\dot{h}\omega ^{2}-\frac{i}{2}\left( \dot{G}+h^{-6}G\bar{F}^{3}%
\frac{\partial }{\partial t}F^{3}\right) F^{3}\Omega \wedge dr+\frac{i}{2}%
\left( \dot{G}+h^{-6}GF^{3}\frac{\partial }{\partial t}\bar{F}^{3}\right) 
\bar{F}^{3}\bar{\Omega}\wedge dr
\end{eqnarray*}%
Then applying the Hodge star: 
\[
\ast \frac{\partial }{\partial t}\psi =4\left( h^{-1}\dot{h}\right)
Gh^{2}dr\wedge \omega +\frac{1}{2}\left( G^{-1}\dot{G}+h^{-6}\bar{F}^{3}%
\frac{\partial }{\partial t}F^{3}\right) F^{3}\Omega +\frac{1}{2}\left(
G^{-1}\dot{G}+h^{-6}F^{3}\frac{\partial }{\partial t}\bar{F}^{3}\right) \bar{%
F}^{3}\bar{\Omega} 
\]%
From this we read off the components $\func{Re}_{1}\psi $, $\func{Im}%
_{1}\psi $ and $\func{Re}_{2}\psi $ as in (\ref{psidot}).

To compute the time derivatives of $\alpha ,\beta ,\gamma ,$ we just
differentiate the expressions (\ref{abc}):%
\begin{eqnarray*}
\dot{\alpha} &=&\frac{\partial }{\partial t}\left( \frac{\theta ^{\prime }}{G%
}\right) \\
&=&\frac{\left( \dot{\theta}\right) ^{\prime }}{G}-\alpha \frac{\dot{G}}{G}
\\
\dot{\beta} &=&\frac{\lambda \dot{\theta}\cos \theta }{h}-\frac{\lambda \dot{%
h}\sin \theta }{h^{2}} \\
&=&\dot{\theta}G^{-1}\left( \gamma -\frac{h^{\prime }}{h}\right) -\frac{\dot{%
h}}{h}\beta \\
\dot{\gamma} &=&\frac{\left( \dot{h}\right) ^{\prime }}{h}-\frac{\dot{h}%
h^{\prime }}{h^{2}}+\frac{\lambda \dot{G}\cos \theta }{h} \\
&&-\frac{\lambda \dot{h}G\cos \theta }{h^{2}}-\frac{\lambda \dot{\theta}%
G\sin \theta }{h} \\
&=&\left( \frac{\dot{h}}{h}\right) ^{\prime }+\left( \frac{\dot{G}}{G}-\frac{%
\dot{h}}{h}\right) \left( \gamma -\frac{h^{\prime }}{h}\right) -\dot{\theta}%
G\beta
\end{eqnarray*}%
and apply the expression (\ref{gamhph}) for $\gamma -\frac{h^{\prime }}{h}$
to get (\ref{abcdot}).
\end{proof}

\subsection{Modified Laplacian coflow}

We will now consider the modified Laplacian coflow of co-closed $G_{2}$%
-stuctures. Let us now assume that $d\psi =0$ and thus $\gamma =0$. We will
write the modified coflow as 
\begin{equation}
\frac{\partial \psi }{\partial t}=\Delta _{\psi }\psi +kd\left( \left( C-%
\func{Tr}T\right) \varphi \right)  \label{modcoflowk}
\end{equation}%
In particular, this flow preserves the condition $d\psi =0$ and hence $%
\gamma =0$. In (\cite{GrigorianCoflow}), this flow was considered only for $%
k=2$. This constant was chosen in order for the linearization of (\ref%
{modcoflowk}) to be the standard Laplacian plus a Lie derivative term.
However, it can be seen from the calculations in (\cite{GrigorianCoflow})
that in fact for any $k>1$ the equation (\ref{modcoflowk}) will be weakly
parabolic in the direction of closed forms and hence the same reasoning can
be using to show short-time existence and uniqueness for the flow (\ref%
{modcoflowk}) with $k>1.$ Also note that the case $k=0$ corresponds to a
standard Laplacian flow.

\begin{theorem}
Suppose we have a co-closed $G_{2}$-structure on $M^{7},$ which is given by (%
\ref{G2warpedstruct2}). Then the flow (\ref{modcoflowk}) is equivalent to
the following evolution equations for warped product parameters $G,h,\theta $%
:%
\begin{subequations}%
\label{lapflowdots}%
\begin{eqnarray}
\frac{\dot{G}}{G} &=&\alpha ^{2}+3\beta ^{2}+k\left( C-\func{Tr}T\right)
\alpha  \nonumber \\
&=&\left( 1-k\right) \alpha ^{2}+3\beta ^{2}+6k\alpha \beta +kC\alpha \\
\dot{\theta} &=&\left( k-1\right) G^{-1}\left( \func{Tr}T\right) ^{\prime } 
\nonumber \\
&=&\left( k-1\right) G^{-1}\left( \alpha -6\beta \right) ^{\prime } \\
\frac{\dot{h}}{h} &=&-\beta \left( \alpha -3\beta \right) -k\left( C-\func{Tr%
}T\right) \beta  \nonumber \\
&=&3\left( 1-2k\right) \beta ^{2}+\left( k-1\right) \alpha \beta -kC\beta
\end{eqnarray}%
\end{subequations}
where $\alpha $ and $\beta $ are given by (\ref{abc}). Moreover, the
evolution of $\alpha ,\beta ,\gamma $ is given by 
\begin{subequations}%
\label{lapflowabcdots} 
\begin{eqnarray}
\dot{\alpha} &=&\left( k-1\right) G^{-2}\left( -\frac{G^{\prime }}{G}\left( 
\func{Tr}T\right) ^{\prime }+\left( \func{Tr}T\right) ^{\prime \prime
}\right) +\left( k-1\right) \alpha ^{3}-6k\beta \alpha ^{2}-3\beta
^{2}\alpha -kC\alpha ^{2} \\
\dot{\beta} &=&\left( k-1\right) G^{-2}\left( \func{Tr}T\right) ^{\prime
}\left( \frac{\beta ^{\prime }}{\alpha +\beta }\right) +\left( 1-k\right)
\beta ^{2}\alpha +3\left( 2k-1\right) \beta ^{3}+kC\beta ^{2}
\end{eqnarray}%
\end{subequations}
where $\func{Tr}T=\alpha -6\beta $.
\end{theorem}

\begin{proof}
We already know the components of $\Delta \psi =\ast \Delta \varphi $ and $%
\frac{\partial \psi }{\partial t}$ from Corollary \ref{CorrLapgam0} and
Lemma \ref{LemPsidot}, respectively. So we just need decompose the
additional part $kd\left( \left( C-\func{Tr}T\right) \varphi \right) $ into $%
\func{Re}_{1},$ $\func{Im}_{1}$ and $\func{Re}_{2}$ components. Consider 
\[
d\left( \left( C-\func{Tr}T\right) \varphi \right) 
\]%
Thus, 
\begin{eqnarray*}
\func{Re}_{1}\left( \left( C-\func{Tr}T\right) \varphi \right) &=&\left( C-%
\func{Tr}T\right) \\
\func{Im}_{1}\left( \left( C-\func{Tr}T\right) \varphi \right) &=&0 \\
\func{Re}_{2}\left( \left( C-\func{Tr}T\right) \varphi \right) &=&\left( C-%
\func{Tr}T\right)
\end{eqnarray*}%
Hence, using Proposition \ref{Propdchi} with $\func{Re}A=C-\func{Tr}T,$ $%
\func{Im}A=0$ and $B=C-\func{Tr}T,$ we get 
\begin{eqnarray*}
\func{Re}_{1}\left( \ast d\left( \left( C-\func{Tr}T\right) \varphi \right)
\right) &=&G^{-1}\left( \func{Im}A^{\prime }+3h^{-1}h^{\prime }\func{Im}%
A+\theta ^{\prime }\func{Re}A-3\lambda BGh^{-1}\sin \theta \right) \\
&=&\left( C-\func{Tr}T\right) G^{-1}\left( \theta ^{\prime }-3\lambda
Gh^{-1}\sin \theta \right) \\
&=&\left( C-\func{Tr}T\right) \left( \alpha -3\beta \right) \\
\func{Im}_{1}\left( \ast d\left( \left( C-\func{Tr}T\right) \varphi \right)
\right) &=&G^{-1}\left( -\func{Re}A^{\prime }-3h^{-1}h^{\prime }\func{Re}%
A+\theta ^{\prime }\func{Im}A-3\lambda BGh^{-1}\cos \theta \right) \\
&=&G^{-1}\left( \func{Tr}T\right) ^{\prime }-3\left( C-\func{Tr}T\right)
\left( G^{-1}\frac{h^{\prime }}{h}-\lambda h^{-1}\cos \theta \right) \\
&=&G^{-1}\left( \func{Tr}T\right) ^{\prime }-3\left( C-\func{Tr}T\right)
\gamma \\
&=&G^{-1}\left( \func{Tr}T\right) ^{\prime } \\
\func{Re}_{2}\left( \ast d\left( \left( C-\func{Tr}T\right) \varphi \right)
\right) &=&-4\lambda h^{-1}\left( \sin \theta \func{Re}A+\cos \theta \func{Im%
}A\right) \\
&=&-4\left( C-\func{Tr}T\right) \beta
\end{eqnarray*}%
where we have also used the fact that $\gamma =0.$ Now also using Corollary %
\ref{CorrLapgam0} we can write 
\begin{subequations}
\begin{eqnarray}
\func{Re}_{1}\left( \Delta \varphi +k\ast d\left( \left( C-\func{Tr}T\right)
\varphi \right) \right) &=&\alpha ^{2}-3\beta \alpha +12\beta ^{2}+k\left( C-%
\func{Tr}T\right) \left( \alpha -3\beta \right) \\
\func{Im}_{1}\left( \Delta \varphi +k\ast d\left( \left( C-\func{Tr}T\right)
\varphi \right) \right) &=&G^{-1}\left( 6\beta ^{\prime }-\alpha ^{\prime
}\right) +kG^{-1}\left( \func{Tr}T\right) ^{\prime } \\
&=&\left( k-1\right) G^{-1}\left( \func{Tr}T\right) ^{\prime } \\
\func{Re}_{2}\left( \Delta \varphi +k\ast d\left( \left( C-\func{Tr}T\right)
\varphi \right) \right) &=&-4\beta \left( \alpha -3\beta \right) -4k\left( C-%
\func{Tr}T\right) \beta
\end{eqnarray}%
\end{subequations}%
Therefore, from Lemma \ref{LemPsidot}, the flow (\ref{modcoflowk}) is
equivalent to 
\begin{subequations}
\begin{eqnarray}
G^{-1}\dot{G}+3h^{-1}\dot{h} &=&\alpha ^{2}-3\beta \alpha +12\beta
^{2}+k\left( C-\func{Tr}T\right) \left( \alpha -3\beta \right) \\
\dot{\theta} &=&\left( k-1\right) G^{-1}\left( \func{Tr}T\right) ^{\prime }
\\
4h^{-1}\dot{h} &=&-4\beta \left( \alpha -3\beta \right) -4k\left( C-\func{Tr}%
T\right) \beta
\end{eqnarray}%
\end{subequations}%
This immediately gives the expressions (\ref{lapflowdots}). To get the
expressions (\ref{lapflowabcdots}) for $\dot{\alpha}$ and $\dot{\beta}$, we
just substitute the expressions for $\dot{G}$,$\dot{h}$ and $\dot{\theta}$
into (\ref{abcdot}) with $\gamma =0$:%
\begin{eqnarray*}
\dot{\alpha} &=&\frac{\left( \dot{\theta}\right) ^{\prime }}{G}-\alpha \frac{%
\dot{G}}{G} \\
&=&-\left( k-1\right) G^{-3}G^{\prime }\left( \func{Tr}T\right) ^{\prime
}+\left( k-1\right) G^{-2}\left( \func{Tr}T\right) ^{\prime \prime }-\alpha
^{3}-3\alpha \beta ^{2}-k\left( C-\func{Tr}T\right) \alpha ^{2} \\
&=&\left( k-1\right) G^{-2}\left( -\frac{G^{\prime }}{G}\left( \func{Tr}%
T\right) ^{\prime }+\left( \func{Tr}T\right) ^{\prime \prime }\right)
+\left( k-1\right) \alpha ^{3}-6k\beta \alpha ^{2}-3\beta ^{2}\alpha
-kC\alpha ^{2} \\
\dot{\beta} &=&\dot{\theta}G^{-1}\left( \frac{\beta ^{\prime }}{\alpha
+\beta }\right) -\frac{\dot{h}}{h}\beta \\
&=&\left( k-1\right) G^{-2}\left( \func{Tr}T\right) ^{\prime }\left( \frac{%
\beta ^{\prime }}{\alpha +\beta }\right) +\beta ^{2}\left( \alpha -3\beta
\right) +k\left( C-\func{Tr}T\right) \beta ^{2} \\
&=&\left( k-1\right) G^{-2}\left( \func{Tr}T\right) ^{\prime }\left( \frac{%
\beta ^{\prime }}{\alpha +\beta }\right) +\left( 1-k\right) \beta ^{2}\alpha
+3\left( 2k-1\right) \beta ^{3}+kC\beta ^{2}
\end{eqnarray*}
\end{proof}

\begin{remark}
Let us compare (\ref{lapflowdots}) for $k=0$ and $k=2.$ For $k=0,$ we have 
\begin{subequations}%
\label{Gthdots0} 
\begin{eqnarray}
\frac{\dot{G}}{G} &=&\alpha ^{2}+3\beta ^{2} \\
\dot{\theta} &=&-G^{-1}\left( \alpha -6\beta \right) ^{\prime } \\
\frac{\dot{h}}{h} &=&3\beta ^{2}-\alpha \beta
\end{eqnarray}%
\end{subequations}%
For $k=2,$ we have we have 
\begin{subequations}%
\label{Gthdots2} 
\begin{eqnarray}
\frac{\dot{G}}{G} &=&-\alpha ^{2}+3\beta ^{2}+12\alpha \beta +2C\alpha \\
\dot{\theta} &=&G^{-1}\left( \alpha -6\beta \right) ^{\prime } \\
\frac{\dot{h}}{h} &=&-9\beta ^{2}+\alpha \beta -2C\beta
\end{eqnarray}%
\end{subequations}%
Note that $\alpha =G^{-1}\theta ^{\prime },$ so the leading order terms for $%
k=2$ in (\ref{Gthdots2}) actually enter with the opposite sign compared to
the $k=0$ case in (\ref{Gthdots0}). However, $k=0$ corresponds to the
Laplacian coflow $\frac{\partial \psi }{\partial t}=\Delta \psi $, so the
\textquotedblleft reverse\textquotedblright\ Laplacian coflow $\frac{%
\partial \psi }{\partial t}=-\Delta \psi $ (which is what was actually
considered by Karigiannis, McKay and Tsui in \cite{KarigiannisMcKayTsui})
would have the same signs on the leading order terms as our modified coflow (%
\ref{modcoflowk}) with $k=2.$ Why this happens is clarified if we consider
the decomposition of $\Delta \varphi $ according to $G_{2}$-representations.
\end{remark}

\begin{lemma}
\label{LemLapG2decom}If $\varphi $ is co-closed warped $G_{2}$-structure
given by (\ref{G2warpedstruct2}), then we can write $\Delta \varphi
=X\lrcorner \psi +\mathrm{i}_{\varphi }\left( s\right) $ with 
\begin{subequations}%
\begin{eqnarray}
X &=&G^{-2}\left( 6\beta ^{\prime }-\alpha ^{\prime }\right) \frac{\partial 
}{\partial r} \\
s &=&-2\alpha \left( \alpha +3\beta \right) G^{2}dr^{2}+\left( \alpha
^{2}-3\beta \alpha +12\beta ^{2}\right) h^{2}g_{6}.
\end{eqnarray}%
\end{subequations}%
\end{lemma}

\begin{proof}
Proposition \ref{PropChiG2cpts}, we know that if we write 
\[
\Delta \varphi =X\lrcorner \psi +\mathrm{i}_{\varphi }\left( s\right) 
\]%
then,%
\begin{eqnarray}
X &=&\left( \func{Im}_{1}\Delta \varphi \right) G^{-1}\frac{\partial }{%
\partial r} \\
s &=&\left( 3\func{Re}_{2}\Delta \varphi -2\func{Re}_{1}\Delta \varphi
\right) G^{2}dr^{2}+\left( \func{Re}_{1}\Delta \varphi \right) h^{2}g_{6}
\end{eqnarray}%
From (\ref{lapcptsgam0}) we have 
\begin{eqnarray*}
\func{Re}_{1}\left( \Delta \varphi \right) &=&\alpha ^{2}-3\beta \alpha
+12\beta ^{2} \\
\func{Im}_{1}\left( \Delta \varphi \right) &=&G^{-1}\left( 6\beta ^{\prime
}-\alpha ^{\prime }\right) \\
\func{Re}_{2}\left( \Delta \varphi \right) &=&-4\beta \left( \alpha -3\beta
\right)
\end{eqnarray*}%
Hence, 
\begin{eqnarray*}
X &=&G^{-2}\left( 6\beta ^{\prime }-\alpha ^{\prime }\right) \frac{\partial 
}{\partial r} \\
s &=&-2\alpha \left( \alpha +3\beta \right) G^{2}dr^{2}+\left( \alpha
^{2}-3\beta \alpha +12\beta ^{2}\right) h^{2}g_{6}.
\end{eqnarray*}
\end{proof}

Lemma \ref{LemLapG2decom} shows that the only second-order derivative terms
(given by $\alpha ^{\prime }$) of the basic variables $G,h,\theta $ occur in 
$\pi _{7}\Delta \varphi $. The $\mathbf{1\oplus 27}$ part of $\Delta \varphi
,$ and hence of $\Delta \psi =\ast \Delta \varphi $, only involves first
derivatives. In general, as it was shown in \cite{GrigorianCoflow}, for
co-closed $G_{2}$-structures, $\pi _{1\oplus 27}\Delta \psi $ has a positive
definite symbol, while $\pi _{7}\Delta \psi $ is negative definite. In the
warped product case, since only $\pi _{7}\Delta \psi $ has leading order
terms, $-\Delta \psi $ has the correct sign, and this was the reason why the
flow $\frac{\partial \psi }{\partial t}=-\Delta \psi $ was used in \cite%
{KarigiannisMcKayTsui}. In a general setting however, both $\Delta \psi $
and $-\Delta \psi $ would have indefinite symbols.

The evolution equations (\ref{Gthdots2}) are in general difficult to
analyze. However we can make some progress in special cases. In particular
suppose $\lambda =0$, so that the underlying $6$-dimensional space is
Calabi-Yau. Then, $\beta =0$. Also suppose that $C=0$. Then, (\ref{Gthdots2}%
) simplifies to the following system:~%
\begin{subequations}%
\label{gthdotsCY}%
\begin{eqnarray}
\dot{G} &=&-G\alpha ^{2} \\
\dot{\theta} &=&G^{-1}\alpha ^{\prime } \\
\theta ^{\prime } &=&G\alpha
\end{eqnarray}%
\end{subequations}%
We can attempt to find separable solutions here. So let, 
\begin{eqnarray*}
G\left( t,r\right) &=&G_{t}\left( t\right) G_{r}\left( r\right) \\
\theta \left( t,r\right) &=&\theta _{t}\left( t\right) \theta _{r}\left(
r\right) \\
\alpha \left( t,r\right) &=&\frac{\theta ^{\prime }}{G}=\frac{\theta _{t}}{%
G_{t}}\frac{\theta _{r}^{\prime }}{G_{r}}
\end{eqnarray*}%
We'll let $\alpha _{r}\left( r\right) =\frac{\theta _{r}^{\prime }}{G_{r}}$
and $\alpha _{t}\left( t\right) =\frac{\theta _{t}}{G_{t}}.$ The equations (%
\ref{gthdotsCY}) then become%
\begin{eqnarray*}
\frac{\dot{G}_{t}}{\alpha _{t}^{2}G_{t}} &=&-\alpha _{r}^{2}=-\lambda _{1} \\
\frac{G_{t}\dot{\theta}_{t}}{\alpha _{t}} &=&G_{r}^{-1}\theta
_{r}^{-1}\alpha _{r}^{\prime }=\lambda _{2}
\end{eqnarray*}%
where $\lambda _{1}\geq 0$ and $\lambda _{2}$ are constants. Hence we find
that 
\begin{eqnarray*}
\alpha _{r}^{2} &=&\frac{\theta _{r}^{\prime }}{G_{r}}=\lambda _{1} \\
\frac{G_{t}\dot{G}_{t}}{\theta _{t}^{2}} &=&-\lambda _{1} \\
\alpha _{r}^{\prime } &=&\lambda _{2}G_{r}\theta _{r} \\
G_{t}\dot{\theta}_{t} &=&\lambda _{2}\alpha _{t}
\end{eqnarray*}%
However, note that $\alpha _{r}^{2}=\lambda _{1}$ is constant, so $\alpha
_{r}^{\prime }=0.$ Hence either $\lambda _{2}=0$ or $\theta _{r}=0$ (trivial
solution). Thus, let $\lambda _{2}=0$. In this case, we find that $\theta
_{t}$ is constant and 
\[
G_{t}^{2}=1-2\lambda _{1}\theta _{t}^{2}t 
\]%
where without loss of generality we set $G_{t}\left( 0\right) =1$. When $%
\lambda _{1}=0$ or $\theta _{t}=0$ we get a trivial solution, otherwise $%
\lambda _{1}$ has to be positive. In this case, we see that the solution
only exists for $t\in \lbrack 0,T)$ where 
\[
T=\frac{1}{2\lambda _{1}\theta _{t}}. 
\]%
Note that whenever the solution exists, from (\ref{torsabc}) we know that
the torsion of the $G_{2}$-structure is proportional to $\alpha $, which is
given by 
\[
\alpha \left( t,r\right) =\alpha _{t}\alpha _{r}=\frac{\lambda _{1}^{\frac{1%
}{2}}\theta _{t}}{\sqrt{1-2\lambda _{1}\theta _{t}^{2}t}} 
\]%
Hence, the torsion increases monotonically until it blows up at $t=T$.

\section{Soliton solutions}

\setcounter{equation}{0}\label{secsoliton}Soliton solutions of geometric
flows are solutions which evolve by diffeomorphisms and scalings. Therefore,
a smooth family $\psi \left( t\right) $ of $G_{2}$-structure $4$-forms would
be a soliton if 
\[
\frac{\partial \psi \left( t\right) }{\partial t}=\mathcal{L}_{X}\psi +4\mu
\psi 
\]%
for some vector field $X$ and a constant $\mu $ (the factor of $4$ is for
later convenience). If moreover we impose the condition that $\psi \left(
t\right) $ is a family of co-closed $G_{2}$-structures, with $d\psi \left(
t\right) =0$ for all $t$, then we would have 
\begin{equation}
\frac{\partial \psi \left( t\right) }{\partial t}=d\left( X\lrcorner \psi
\right) +4\mu \psi  \label{solitonevol}
\end{equation}%
In the case of a warped product $G_{2}$-structure, any vector $X$ that
respects the symmetry of the space would have to be proportional to $\frac{%
\partial }{\partial r}$ and only have dependence on the coordinate $r$. In
particular, we could write 
\begin{equation}
X=l\left( r\right) G^{-1}\frac{\partial }{\partial r}  \label{Xl}
\end{equation}%
for some function $l\left( r\right) $ on $L$.

\begin{lemma}
\label{SolFlow}Under the flow (\ref{solitonevol}), the warped product
parameters $G,h,\theta $ satisfy the following evolution equation 
\begin{subequations}%
\label{Ghtsoliton}%
\begin{eqnarray}
G^{-1}\dot{G} &=&G^{-1}l^{\prime }+\mu \\
\dot{\theta} &=&\alpha l \\
h^{-1}\dot{h} &=&-\frac{G^{-1}\beta ^{\prime }l}{a+\beta }+\mu
\end{eqnarray}%
\end{subequations}%
Moreover the quantities $\alpha ,\beta ,\gamma $ satisfy the following
equations 
\begin{subequations}%
\label{abcsoliton} 
\begin{eqnarray}
\dot{\alpha} &=&G^{-1}\alpha ^{\prime }l-\alpha \mu \\
\dot{\beta} &=&G^{-1}l\beta ^{\prime }-\beta \mu \\
\dot{\gamma} &=&G^{-1}l\left( \left( \frac{\beta ^{\prime }}{\alpha +\beta }%
\right) ^{2}-\left( \frac{\beta ^{\prime }}{\alpha +\beta }\right) ^{\prime
}-G^{2}\alpha \beta +\frac{G^{\prime }}{G}\left( \frac{\beta ^{\prime }}{%
\alpha +\beta }\right) \right)
\end{eqnarray}%
\end{subequations}%
\end{lemma}

\begin{proof}
We will consider the components $\func{Re}_{1},$ $\func{Im}_{1}$ and $\func{%
Re}_{2}$ of the right-hand side of (\ref{solitonevol}). First consider $%
d\left( l\left( r\right) G^{-1}\frac{\partial }{\partial r}\lrcorner \psi
\right) $. We have%
\begin{eqnarray*}
lG^{-1}\frac{\partial }{\partial r}\lrcorner \psi &=&lG^{-1}\frac{\partial }{%
\partial r}\lrcorner \left( \frac{1}{2}h^{4}\omega ^{2}-\frac{iGF^{3}}{2}%
\Omega \wedge dr+\frac{iG\bar{F}^{3}}{2}\bar{\Omega}\wedge dr\right) \\
&=&\frac{1}{2}\left( il\right) F^{3}\Omega -\frac{1}{2}\left( il\right) \bar{%
F}^{3}\bar{\Omega}
\end{eqnarray*}%
So, 
\begin{eqnarray*}
\func{Re}_{1}\left( lG^{-1}\frac{\partial }{\partial r}\lrcorner \psi
\right) &=&0 \\
\func{Im}_{1}\left( lG^{-1}\frac{\partial }{\partial r}\lrcorner \psi
\right) &=&l \\
\func{Re}_{2}\left( lG^{-1}\frac{\partial }{\partial r}\lrcorner \psi
\right) &=&0
\end{eqnarray*}%
Thus, from Proposition \ref{Propdchi}, we obtain 
\begin{eqnarray*}
\func{Re}_{1}\left( \ast d\left( lG^{-1}\frac{\partial }{\partial r}%
\lrcorner \psi \right) \right) &=&G^{-1}\left( \func{Im}A^{\prime
}+3h^{-1}h^{\prime }\func{Im}A+\theta ^{\prime }\func{Re}A-3\lambda
BGh^{-1}\sin \theta \right) \\
&=&G^{-1}\left( l^{\prime }+3\frac{h^{\prime }}{h}l\right) \\
&=&G^{-1}\left( l^{\prime }-\frac{3\beta ^{\prime }l}{a+\beta }\right) \\
\func{Im}_{1}\left( \ast d\left( lG^{-1}\frac{\partial }{\partial r}%
\lrcorner \psi \right) \right) &=&G^{-1}\left( -\func{Re}A^{\prime
}-3h^{-1}h^{\prime }\func{Re}A+\theta ^{\prime }\func{Im}A-3\lambda
BGh^{-1}\cos \theta \right) \\
&=&G^{-1}\theta ^{\prime }l \\
&=&\alpha l \\
\func{Re}_{2}\left( \ast d\left( lG^{-1}\frac{\partial }{\partial r}%
\lrcorner \psi \right) \right) &=&-4\lambda h^{-1}\left( \sin \theta \func{Re%
}A+\cos \theta \func{Im}A\right) \\
&=&-4\lambda h^{-1}l\cos \theta \\
&=&-\frac{4G^{-1}\beta ^{\prime }l}{a+\beta }
\end{eqnarray*}%
Now if, 
\[
\frac{\partial \psi }{\partial t}=d\left( lG^{-1}\frac{\partial }{\partial r}%
\lrcorner \psi \right) +4\mu \psi 
\]%
then, 
\begin{eqnarray*}
G^{-1}\dot{G}+3h^{-1}\dot{h} &=&G^{-1}\left( l^{\prime }-\frac{3\beta
^{\prime }l}{\alpha +\beta }\right) +4\mu \\
\dot{\theta} &=&\alpha l \\
4h^{-1}\dot{h} &=&-\frac{4G^{-1}\beta ^{\prime }l}{\alpha +\beta }+4\mu
\end{eqnarray*}%
From this we obtain (\ref{Ghtsoliton}). To find the evolution of $\alpha
,\beta ,\gamma $ we just substitute (\ref{Ghtsoliton}) into the expressions (%
\ref{abcdot}) and set $\gamma =0$: 
\begin{eqnarray*}
\dot{\alpha} &=&\frac{\left( \dot{\theta}\right) ^{\prime }}{G}-\alpha \frac{%
\dot{G}}{G} \\
&=&G^{-1}\left( \alpha l\right) ^{\prime }-G^{-1}\alpha l^{\prime }-\alpha
\mu \\
&=&G^{-1}\alpha ^{\prime }l-\alpha \mu \\
\dot{\beta} &=&\dot{\theta}G^{-1}\left( \gamma -\frac{h^{\prime }}{h}\right)
-\frac{\dot{h}}{h}\beta \\
&=&\alpha lG^{-1}\left( \frac{\beta ^{\prime }}{\alpha +\beta }\right) +%
\frac{G^{-1}\beta ^{\prime }\beta l}{a+\beta }-\mu \beta \\
&=&G^{-1}l\beta ^{\prime }-\beta \mu \\
\dot{\gamma} &=&\left( \frac{\dot{h}}{h}\right) ^{\prime }+\left( \frac{\dot{%
G}}{G}-\frac{\dot{h}}{h}\right) \left( \gamma -\frac{h^{\prime }}{h}\right) -%
\dot{\theta}G\beta \\
&=&-\left( \frac{G^{-1}\beta ^{\prime }l}{a+\beta }\right) ^{\prime
}+G^{-1}\left( l^{\prime }+\frac{\beta ^{\prime }l}{a+\beta }\right) \left( 
\frac{\beta ^{\prime }}{\alpha +\beta }\right) -G\alpha \beta l \\
&=&G^{-1}l\left( \left( \frac{\beta ^{\prime }}{\alpha +\beta }\right)
^{2}-\left( \frac{\beta ^{\prime }}{\alpha +\beta }\right) ^{\prime
}-G^{2}\alpha \beta +\frac{G^{\prime }}{G}\left( \frac{\beta ^{\prime }}{%
\alpha +\beta }\right) \right)
\end{eqnarray*}
\end{proof}

Now suppose we want soliton solutions of the modified Laplacian coflow (\ref%
{modcoflowk}). Then at every time $t$ we require%
\begin{equation}
\Delta _{\psi }\psi +kd\left( \left( C-\func{Tr}T\right) \varphi \right)
=d\left( lG^{-1}\frac{\partial }{\partial r}\lrcorner \psi \right) +4\mu \psi
\label{coflowsoliton}
\end{equation}%
This is now a time-independent equation, so we can redefine the coordinate $%
r $ on $L$ such that $G=1.$ Then by equating (\ref{Ghtsoliton}) with (\ref%
{lapflowdots}) we get the following equations for $h,\theta ,l$ and $\mu $:

\begin{proposition}
The soliton solutions of the modified Laplacian coflow (\ref{modcoflowk})
satisfy the following equations 
\begin{subequations}%
\label{soleq}%
\begin{eqnarray}
\left( 1-k\right) \alpha ^{2}+3\beta ^{2}+6k\alpha \beta +kC\alpha
&=&l^{\prime }+\mu \\
\left( k-1\right) \left( \alpha -6\beta \right) ^{\prime } &=&\alpha l \\
3\left( 1-2k\right) \beta ^{2}+\left( k-1\right) \alpha \beta -kC\beta &=&-%
\frac{\beta ^{\prime }l}{a+\beta }+\mu  \label{soleq3}
\end{eqnarray}%
\end{subequations}%
where $\alpha =\theta ^{\prime }$ and $\beta =\frac{\lambda \sin \theta }{h}$%
.
\end{proposition}

Note that equivalently, we could have equated the equations for $\dot{\alpha}%
,\dot{\beta}$ and $\dot{\gamma}.$ The result would be an equivalent set of
equations, however $\dot{\gamma}=0$ equation gives us 
\begin{equation}
\left( \frac{\beta ^{\prime }}{\alpha +\beta }\right) ^{2}-\left( \frac{%
\beta ^{\prime }}{\alpha +\beta }\right) ^{\prime }-\alpha \beta =0
\label{gamdot0}
\end{equation}%
Even though this can still be obtained from (\ref{soleq}), this explicitly
gives us the co-closed condition (in the nearly K\"{a}hler case, when $%
\lambda \neq 0$).

\subsection{Calabi-Yau case}

If $M^{6}$ is a Calabi-Yau space, then $\lambda =0$, so $\beta =0$. From (%
\ref{soleq3}) this immediately gives $\mu =0$. Also, note that since we are
imposing the condition for the $G_{2}$-structure to be co-closed, we have $%
\gamma =0$. Since $\lambda =0$, this then gives $h^{\prime }=0.$ Hence,
without loss of generality we can assume that $h=1$. Since both $G$ and $h$
are equal to $1$, the metric on $L\times M^{6}$ is just the product metric,
and the $G_{2}$-structure $3$-form $\varphi $ is given by 
\[
\varphi =\frac{1}{2}e^{i\theta }\Omega +\frac{1}{2}e^{-i\theta }\bar{\Omega}%
+dr\wedge \omega 
\]%
where $\theta $ is a function of $r$. The torsion tensor is then given by 
\[
T=\func{diag}\left( \alpha ,0,...,0\right) . 
\]%
The torsion tensor of the $G_{2}$-structure The other equations also
simplify. Thus we have the following special case of (\ref{soleq}).

\begin{corollary}
Suppose $\lambda =0$, so that the underlying $6$-manifold is Calabi-Yau.
Then, the soliton solutions of the modified Laplacian coflow with $k=2$
satisfy 
\begin{subequations}%
\label{CYsoleqs}%
\begin{eqnarray}
l^{\prime } &=&-\alpha ^{2}+2C\alpha \\
\alpha ^{\prime } &=&\alpha l  \label{cyalphprime} \\
\mu &=&0
\end{eqnarray}%
\end{subequations}%
.
\end{corollary}

Define the quantity $R$ via%
\begin{equation}
R^{2}=l^{2}+\left( \alpha -2C\right) ^{2}  \label{I1int}
\end{equation}%
It follows immediately from (\ref{CYsoleqs}) that $R^{\prime }=0$.
Therefore, $R$ is a first integral for the system (\ref{CYsoleqs}). Recall
that the torsion of the the warped product $G_{2}$-structure is proportional
to $\alpha $ in this case. The expression (\ref{I1int}) shows that the only
way we can have a torsion-free solution (i.e. $\alpha =0$) is if $l$ is
constant. Also, if $R=0$, we get constant solutions $l=0$ and $\alpha =$ $2C$%
. So suppose $R$ is non-zero. This enables us to solve the equations (\ref%
{CYsoleqs}). Indeed, from (\ref{cyalphprime}) we get 
\begin{equation}
r=\pm \int \frac{d\alpha }{\alpha \sqrt{R^{2}-\left( \alpha -2C\right) ^{2}}}
\label{rintalpha}
\end{equation}%
After integrating (\ref{rintalpha}), we will obtain $\alpha \left( r\right) $
and we will then use it to find $l$ and $\theta :$ 
\begin{subequations}%
\label{cylthetaalpha} 
\begin{eqnarray}
l &=&\frac{\alpha ^{\prime }}{\alpha } \\
\theta &=&\int \alpha \left( r\right) dr
\end{eqnarray}%
\end{subequations}%
Since Calabi-Yau $3$-form $\Omega $ is only defined up to a constant phase
factor, we will neglect the constant of integration when computing $\theta $%
, since we can always redefine $\theta $ by a translation. The actual
solutions that we obtain will depend on the sign of the quantity $%
R^{2}-4C^{2}.$ We summarize the findings in Theorem and after the statement
we will individually consider the three different cases: $R^{2}-4C^{2}=0,$ $%
R^{2}-4C^{2}>0$ and $R^{2}-4C^{2}<0.$

\begin{theorem}
\label{ThmCYSols}Let $M^{7}=L\times M^{6}$ where $L$ is a $1$-dimensional
space diffeomorphic to $\mathbb{R}\ $and $M^{6}$ a Calabi-Yau $3$-fold. The
given the soliton equation for the modified Laplacian coflow 
\[
\Delta _{\psi }\psi +2d\left( \left( C-\func{Tr}T\right) \varphi \right)
=d\left( lG^{-1}\frac{\partial }{\partial r}\lrcorner \psi \right) +4\mu
\psi 
\]%
together with initial conditions 
\begin{equation}
\left\{ 
\begin{array}{r}
\alpha \left( r_{0}\right) =\alpha _{0} \\ 
l\left( r_{0}\right) =0 \\ 
\theta _{0}\left( r_{0}\right) =\theta _{0}%
\end{array}%
\right.
\end{equation}%
for some $r_{0}\in \mathbb{R}$ are: 
\begin{equation}
\left\{ 
\begin{array}{r}
G=1 \\ 
h=1 \\ 
\mu =0%
\end{array}%
\right.  \label{cysolGhmu}
\end{equation}%
and

\begin{enumerate}
\item If $\alpha _{0}=4C$, 
\begin{equation}
\left\{ 
\begin{array}{c}
\alpha =\frac{4C}{4C^{2}\tilde{r}^{2}+1} \\ 
l=-\frac{8C^{2}\tilde{r}}{4C^{2}\tilde{r}^{2}+1} \\ 
\theta =2\arctan \left( 2C\tilde{r}\right) +\theta _{0}%
\end{array}%
\right.
\end{equation}%
where $\tilde{r}=r-r_{0}.$

\item If $\alpha _{0}=2C\pm R$ for $\left\vert R\right\vert >2C,$ 
\begin{equation}
\left\{ 
\begin{array}{c}
\alpha =\frac{\pm 2RQ^{2}e^{-\tilde{r}Q}}{\left( e^{-\tilde{r}Q}R\mp
2C\right) ^{2}+Q^{2}} \\ 
l=\frac{R^{2}Q\left( e^{-2\tilde{r}Q}-1\right) }{\left( e^{-\tilde{r}Q}R\mp
2C\right) ^{2}+Q^{2}} \\ 
\theta =2\arctan \left( \frac{1}{Q}\left( 2C\mp e^{-\tilde{r}Q}R\right)
\right) +\tilde{\theta}_{0}%
\end{array}%
\right.
\end{equation}%
where $\tilde{r}=r-r_{0}$ and $Q^{2}=R^{2}-4C^{2}.$ Also, $\tilde{\theta}%
_{0}=\theta _{0}-2\arctan \left( \frac{2C\mp R}{Q}\right) .$

\item If $\alpha _{0}=2C\pm R$ for $\left\vert R\right\vert <2C,$%
\begin{equation}
\left\{ 
\begin{array}{c}
\alpha =\frac{Q^{2}}{2C+R\cos \tilde{r}Q} \\ 
l=\frac{QR\sin \left( \tilde{r}Q\right) }{2C+R\cos \left( \tilde{r}Q\right) }
\\ 
\theta =2\arctan \left( \frac{2C-R}{Q}\tan \left( \frac{1}{2}\tilde{r}%
Q\right) \right) +\theta _{0}%
\end{array}%
\right.  \label{CYsolperiodic}
\end{equation}%
where $\tilde{r}=r-r_{0}$ for $\alpha _{0}=2C-R$ and $\tilde{r}=r-r_{0}+%
\frac{\pi }{Q}$ for $\alpha _{0}=2C+R,$ and $Q^{2}=4C^{2}-R^{2}.$
\end{enumerate}

If $L\cong S^{1}$, with $r\in \lbrack 0,2\pi )$, then non-trivial global
solutions exist if and only if $\alpha _{0}=2C\pm R$ for $\left\vert
R\right\vert <2C$ such that 
\[
Q^{2}=4C^{2}-R^{2}=2n 
\]%
and are then (\ref{cysolGhmu}) together with (\ref{CYsolperiodic}).
Moreover, if $C=0,$ then there exists a trivial solution $\alpha =l=0$ and $%
\theta =\theta _{0}$ for some constant $\theta _{0}$.
\end{theorem}

\subsubsection{$R^{2}-4C^{2}=0$}

If $R^{2}-4C^{2}=0$, then the solution for $\alpha $ is 
\begin{equation}
\alpha =\frac{4C}{4C^{2}r^{2}+1}  \label{alpha0}
\end{equation}%
Note that since we can always apply a translation on $L$ to obtain a shifted
coordinate $r,$ we have neglected the constant of integration in (\ref%
{rintalpha}). From (\ref{cylthetaalpha}) we then immediately obtain the
solutions for $l$ and $\theta $:%
\begin{subequations}%
\label{Q0sols} 
\begin{eqnarray}
l &=&-\frac{8C^{2}r}{4C^{2}r^{2}+1} \\
\theta &=&2\arctan \left( 2Cr\right)
\end{eqnarray}%
\end{subequations}%
Note that when $C=0$ this gives a trivial solution $l=\theta =0.$ Figure \ref%
{figcysol0} shows a phase diagram for this solution in the $\alpha -l$ space

\FRAME{ftbphFU}{3.7178in}{3.3261in}{0pt}{\Qcb{Phase diagram for the case $%
R^{2}-4C^{2}=0.$ When $r\longrightarrow -\infty ,$ $\protect\alpha %
,l\longrightarrow 0$ and $\protect\theta \longrightarrow -\protect\pi .$
When $r\longrightarrow +\infty ,$ $\protect\alpha ,l\longrightarrow 0$ and $%
\protect\theta \longrightarrow +\protect\pi .$ }}{\Qlb{figcysol0}}{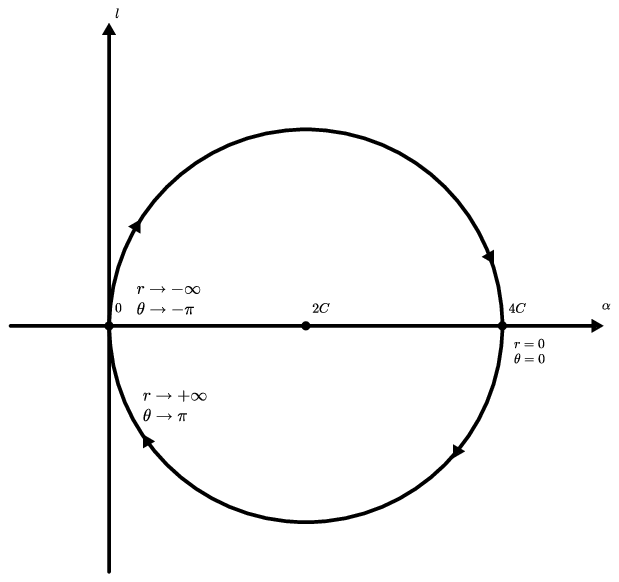%
}{\special{language "Scientific Word";type "GRAPHIC";maintain-aspect-ratio
TRUE;display "USEDEF";valid_file "F";width 3.7178in;height 3.3261in;depth
0pt;original-width 2.5538in;original-height 2.3705in;cropleft "0";croptop
"1";cropright "1";cropbottom "0";filename 'cysol0.eps';file-properties
"XNPEU";}}In particular, from the conserved quantity (\ref{I1int}), we see
that this is the solution that we obtain to the system (\ref{CYsoleqs})
together with the initial conditions 
\begin{equation}
\left\{ 
\begin{array}{r}
\alpha \left( 0\right) =4C \\ 
l\left( 0\right) =0%
\end{array}%
\right. .  \label{cysol0ics}
\end{equation}%
If $L$ is non-compact, then this solution is defined globally. On the other
hand, if $L\cong S^{1},$ then this solution is not globally defined, and
exists only locally.

\subsubsection{$R^{2}-4C^{2}>0$}

If $R^{2}-4C^{2}>0$, then define the quantity $Q$ via%
\begin{equation}
Q^{2}=R^{2}-4C^{2}.
\end{equation}%
Then from (\ref{rintalpha}), we obtain a solution 
\begin{equation}
\alpha =\frac{4A_{0}Q^{2}e^{-rQ}}{\left( A_{0}e^{-rQ}-4C\right) ^{2}+4Q^{2}}
\label{alphaplus}
\end{equation}%
where $A_{0}$ is a constant that depends on initial conditions. Note that an
equivalent solution is obtained by taking $r\longrightarrow -r$. From (\ref%
{alphaplus}) we easily obtain solutions for $l$ and $\theta $:%
\begin{subequations}%
\label{Qpossols} 
\begin{eqnarray}
l &=&\frac{Q\left( A_{0}^{2}e^{-2rQ}-4R^{2}\right) }{\left(
A_{0}e^{-rQ}-4C\right) ^{2}+4Q^{2}} \\
\theta &=&2\arctan \left( \frac{1}{2Q}\left( 4C-A_{0}e^{-rQ}\right) \right)
\end{eqnarray}%
\end{subequations}%
Figure \ref{figcysolpos} shows for this solution in the $\alpha -l$ space.
Similarly as in the case for $R^{2}-4C^{2}=0,$ these solutions are only
defined globally for non-compact $L.$ Taking $A_{0}=\pm 2R,$ and translating
the $r$ coordinate, the solutions (\ref{Qpossols}) are hence solutions of (%
\ref{CYsoleqs}) with the initial conditions%
\begin{equation}
\left\{ 
\begin{array}{r}
\alpha \left( r_{0}\right) =2C\pm R \\ 
l\left( r_{0}\right) =0%
\end{array}%
\right.  \label{cysolposics}
\end{equation}%
for some $r_{0}\in \mathbb{R}$ and $\left\vert R\right\vert >2\left\vert
C\right\vert $. \FRAME{ftbphFU}{5.5953in}{4.1485in}{0pt}{\Qcb{Phase diagram
for the case $Q^{2}=R^{2}-4C^{2}>0.$ The sign of $\protect\alpha $ in the
solution depends on the sign of the constant $A_{0}$. When $r\longrightarrow
-\infty ,$ $\protect\alpha \longrightarrow 0$ and $l\longrightarrow Q.$ When 
$r\longrightarrow +\infty $, $\protect\alpha \longrightarrow 0$ and $%
l\longrightarrow -Q$.}}{\Qlb{figcysolpos}}{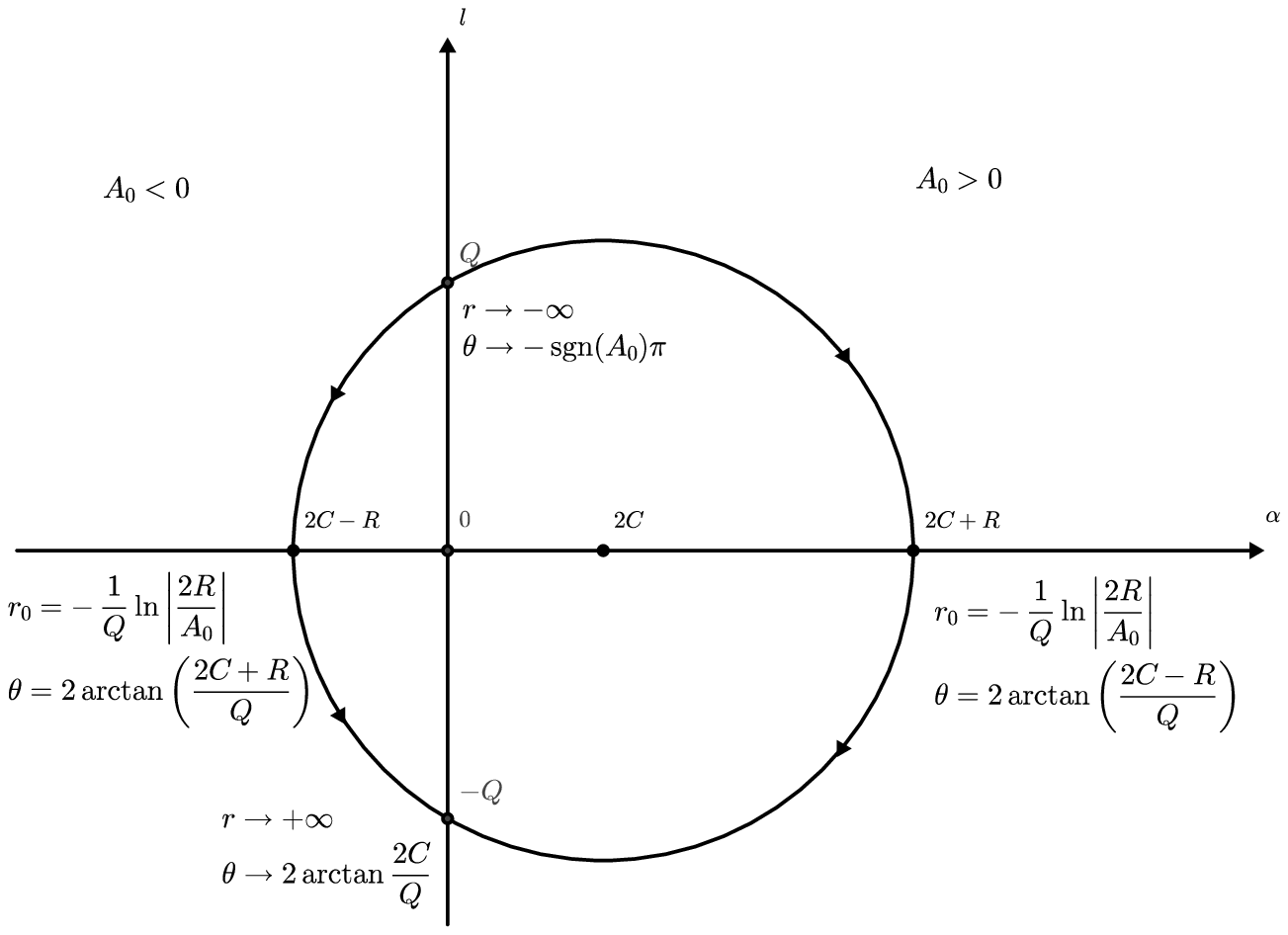}{\special{language
"Scientific Word";type "GRAPHIC";maintain-aspect-ratio TRUE;display
"USEDEF";valid_file "F";width 5.5953in;height 4.1485in;depth
0pt;original-width 7.0396in;original-height 5.1854in;cropleft "0";croptop
"1";cropright "1";cropbottom "0";filename 'cysolPos.eps';file-properties
"XNPEU";}}

In the case when $C=0$, then $R=$ $Q$, so setting $b=-Q\ \ $and $c=-\frac{%
A_{0}}{2Q}$ we obtain 
\begin{eqnarray*}
l &=&b\frac{1-c^{2}e^{2br}}{1+c^{2}e^{2br}} \\
\theta &=&2\arctan \left( ce^{br}\right)
\end{eqnarray*}%
This is precisely the solution that was obtained by Karigiannis, McKay and
Tsui in \cite{KarigiannisMcKayTsui} for the negative Laplacian flow soliton
in this setting.

\subsubsection{ $R^{2}-4C^{2}<0$}

Now suppose $R^{2}-4C^{2}$ is negative. Then we let 
\[
Q^{2}=4C^{2}-R^{2} 
\]%
The solution for $\alpha $ is then 
\[
\alpha =\frac{4Q^{2}A_{0}}{8CA_{0}-\left( 4R^{2}+A_{0}^{2}\right) \cos
rQ-i\left( 4R^{2}-A_{0}^{2}\right) \sin rQ} 
\]%
Since we need $\alpha $ to be real, we have to have $A_{0}^{2}=\pm 4R^{2}.$
Hence, we get the following possible solutions 
\begin{equation}
\alpha =\left\{ 
\begin{array}{r}
\frac{Q^{2}}{2C\pm R\cos rQ}\ \ \text{if }A_{0}^{2}=+4R^{2} \\ 
\frac{Q^{2}}{2C\pm R\sin rQ}\ \text{if }A_{0}^{2}=-4R^{2}%
\end{array}%
\right.  \label{alphaneg}
\end{equation}%
Note that these solutions are equivalent under diffeomorphisms of $L$ - so
by redefining $r$ we can move from one solution to another. Therefore,
without loss of generality we will only consider the solution 
\begin{equation}
\alpha =\frac{Q^{2}}{2C+R\cos rQ}  \label{alphaneg1}
\end{equation}%
From (\ref{alphaneg1}) we obtain the corresponding solutions for $l$ and $%
\theta $:%
\begin{subequations}%
\label{Qnegsols}%
\begin{eqnarray}
l &=&\frac{QR\sin \left( rQ\right) }{2C+R\cos \left( rQ\right) } \\
\theta &=&2\arctan \left( \frac{2C-R}{Q}\tan \left( \frac{1}{2}rQ\right)
\right)
\end{eqnarray}%
\end{subequations}%
Given the freedom to redefine the coordinate $r$, the solutions of the form (%
\ref{Qnegsols}) are solutions to the system (\ref{CYsoleqs}) together with
the initial conditions 
\begin{equation}
\left\{ 
\begin{array}{r}
\alpha \left( r_{0}\right) =2C\pm R \\ 
l\left( r_{0}\right) =0%
\end{array}%
\right.
\end{equation}%
for some $r_{0}\in \mathbb{R}$ and $\left\vert R\right\vert <2\left\vert
C\right\vert $. Figure \ref{figcysolneg} shows the phase diagram for this
solution with $r_{0}=0$. \FRAME{ftbphFU}{4.8273in}{3.7535in}{0pt}{\Qcb{Phase
diagram for the case $4C^{2}-R^{2}>0.$ }}{\Qlb{figcysolneg}}{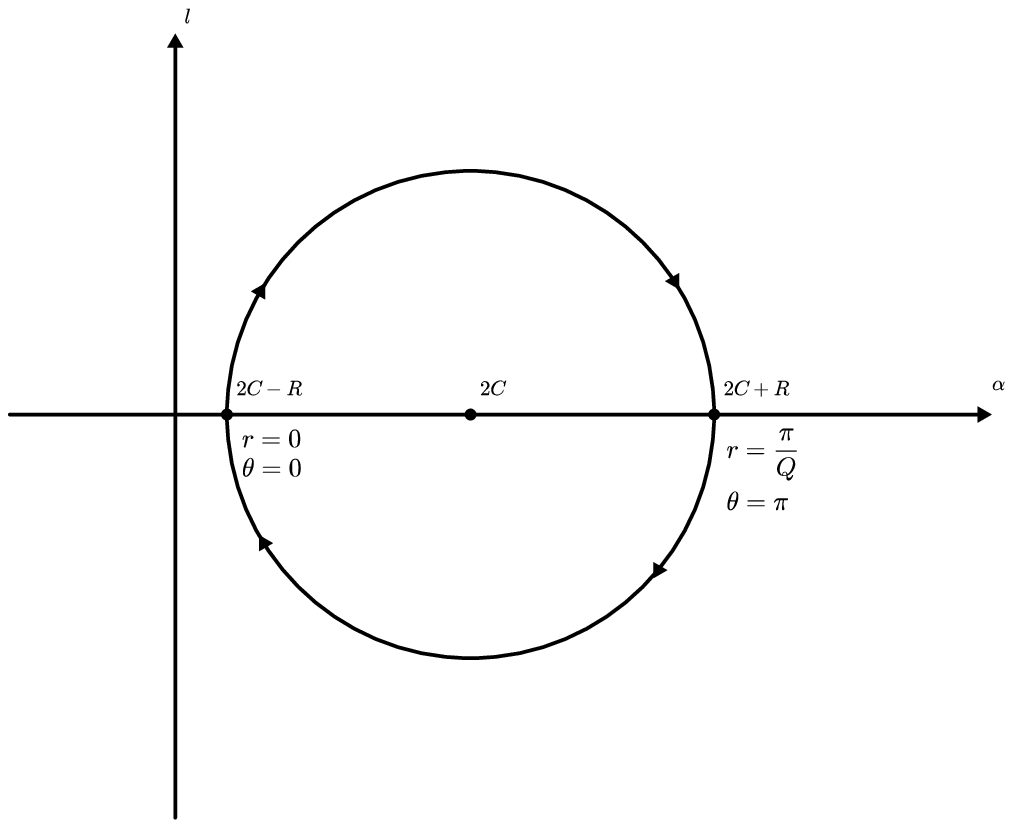}{%
\special{language "Scientific Word";type "GRAPHIC";maintain-aspect-ratio
TRUE;display "USEDEF";valid_file "F";width 4.8273in;height 3.7535in;depth
0pt;original-width 4.7987in;original-height 3.7268in;cropleft "0";croptop
"1";cropright "1";cropbottom "0";filename 'cysolNeg.eps';file-properties
"XNPEU";}}

The solutions (\ref{Qnegsols}) are of particular interest because they are
periodic and thus given appropriate initial conditions they are globally
defined when the manifold $M^{7}=L\times M^{6}$ is compact - in particular
when $L=S^{1}.$ Suppose $r$ is a coordinate on $S^{1}$, taking values in $%
[0,2\pi ).$ Then, if $Q=2n$ for some integer $n$, then the solutions (\ref%
{Qnegsols}) are periodic, and are hence well-defined on $S^{1}$. Note that
when $\frac{1}{2}rQ$ is equal to a half-integer multiple of $\pi ,$ there is
a discontinuity in the solution for $\theta $, with $\theta \longrightarrow
\pi $ when $\frac{1}{2}rQ$ approaches a half-integer multiple of $\pi $ from
the left, and $\theta \longrightarrow -\pi $ when taking the limit from the
right. However $\theta $ is itself defined up to an integer multiple of $%
2\pi $, so the actual solution for the $G_{2}$-structure is still continuous
and moreover smooth.

\subsection{Nearly K\"{a}hler case}

Let us now consider the case when the $6$-dimensional base manifold is
nearly K\"{a}hler. As before, let $k=2$ and $G=1.$ Then, the soliton
equation (\ref{soleq}) become 
\begin{subequations}%
\label{nksoleq}%
\begin{eqnarray}
l^{\prime } &=&-\alpha ^{2}+3\beta ^{2}+12\alpha \beta +2C\alpha -\mu
\label{nksoleql} \\
\left( \alpha -6\beta \right) ^{\prime } &=&\alpha l  \label{nksoleqa} \\
\frac{\beta ^{\prime }l}{\alpha +\beta } &=&9\beta ^{2}-\alpha \beta
+2C\beta +\mu  \label{nksoleqb}
\end{eqnarray}%
\end{subequations}%
This system is very difficult to analyze, because it has no apparent
symmetries or conserved quantities. We can however look at special solutions
where at least one of the variables $\alpha ,\beta ,l$ is constant.

\begin{theorem}
\label{Propconstlsol}The only solutions of the system (\ref{nksoleq}) with
at least one of the variables $\alpha ,\beta ,l$ constant are
\end{theorem}

\begin{enumerate}
\item $\alpha =0,\ \beta =0,~l^{\prime }=-\mu ,$ for arbitrary $\mu $

\item $\alpha =0,$ $\mu =-9\beta ^{2}-2C\beta ,$ $l^{\prime }=2\beta \left(
6\beta +C\right) ,$ for arbitrary constant $\beta $

\item $\alpha =0,$ $\beta =-\frac{1}{6}C,\ \mu =\frac{1}{12}C^{2},$ $l$
arbitrary

\item $l=0,\ \alpha =\frac{1}{10}C\pm \frac{1}{10}\sqrt{C^{2}-10\mu }$, $%
\beta =-\alpha $ \ for $\mu \leq \frac{C^{2}}{10}$

\item $l=0,\ \alpha =4\sqrt{3\mu }+2C,\beta =\frac{1}{3}\sqrt{3\mu }\ $for $%
\mu \geq 0$
\end{enumerate}

\begin{proof}
We will consider different cases where $\beta ,$ $\alpha $ or $l$ are
constant.

\begin{description}
\item[Constant $\protect\beta $] Suppose $\beta ^{\prime }=0$. Then, (\ref%
{nksoleqb}) becomes 
\begin{equation}
\left( 9\beta ^{2}-\alpha \beta +2C\beta +\mu \right) \left( \alpha +\beta
\right) =0  \label{nksoleqb2}
\end{equation}%
From this, either $\mu =0$ and $\beta =0,$ or $\alpha $ must also be
constant and must either satisfy $9\beta ^{2}-\alpha \beta +2C\beta +\mu =0$
or $\alpha +\beta =0$. In the first case, the equations reduce to (\ref%
{CYsoleqs}) - the equations we had in the Calabi-Yau case. Now however,
since $\lambda \neq 0$, in order to have $\beta =\frac{\lambda \sin \theta }{%
h}=0$, we must have $\sin \theta =0.$ In particular, $\theta $ must be
constant. However, from Theorem \ref{ThmCYSols}, this is true if and only if 
$C=0$ and we have a trivial solution. Therefore, for a non-trivial solution, 
$\alpha $ has to be constant. From (\ref{nksoleqa}), this however implies
that $\alpha l=0.$ The case $l=0$ will be considered below. Suppose $\alpha
=0$. Then, from (\ref{nksoleql}) we have 
\[
l^{\prime }=3\beta ^{2}-\mu 
\]%
From (\ref{nksoleqb2}), either $\beta $ is also zero, and $l^{\prime }=-\mu $
or $\mu =-9\beta ^{2}-2C\beta ,$ so that 
\[
l^{\prime }=2\beta \left( 6\beta +C\right) . 
\]%
Thus, we obtain solutions $1$ and $2$.

\item[Constant $\protect\alpha $] Suppose $\alpha ^{\prime }=0$. Then, the
equations (\ref{nksoleq}) become 
\begin{subequations}
\begin{eqnarray}
l^{\prime } &=&-\alpha ^{2}+3\beta ^{2}+12\alpha \beta +2C\alpha -\mu
\label{nksolacl} \\
\beta ^{\prime } &=&-\frac{1}{6}\alpha l  \label{nksolbac} \\
\beta ^{\prime }l &=&\left( 9\beta ^{2}-\alpha \beta +2C\beta +\mu \right)
\left( \alpha +\beta \right)  \label{nksolcac}
\end{eqnarray}%
\end{subequations}%
We already considered the case of constant $\beta $, so suppose $\alpha \neq
0\ $and $l\neq 0$. From equations (\ref{nksolacl}) and (\ref{nksolbac}), we
find a conserved quantity $F,$ given by 
\[
\alpha l^{2}+12\left( \beta ^{3}+6\alpha \beta ^{2}-\left( \alpha
^{2}-2C\alpha +\mu \right) \beta \right) =F 
\]%
while, from equations (\ref{nksolbac}) and (\ref{nksolcac}), we find 
\[
\alpha l^{2}+6\left( 9\beta ^{2}-\alpha \beta +2C\beta +\mu \right) \left(
\alpha +\beta \right) =0. 
\]%
From these two equations, we find that $\beta $ satisfies a cubic equation
with constant coefficients, and hence must also be constant. Therefore, we
do not get any new solutions in this case.

\item[Constant $l$] Suppose $l^{\prime }=0$. Then, the equation (\ref%
{nksoleql}) becomes%
\begin{equation}
-\alpha ^{2}+3\beta ^{2}+12\alpha \beta +2C\alpha -\mu =0  \label{nksoleql2}
\end{equation}%
Differentiating this, we find 
\[
\alpha ^{\prime }\left( -\alpha +6\beta +C\right) +3\beta ^{\prime }\left(
\beta +2\alpha \right) =0 
\]%
Substituting (\ref{nksoleqa}), we get 
\[
\alpha l\left( -\alpha +6\beta +C\right) +3\beta ^{\prime }\left( 13\beta
+2C\right) =0 
\]%
Now using (\ref{nksoleqb}), we get another polynomial equation for $\alpha $
and $\beta $%
\begin{equation}
\alpha l^{2}\left( -\alpha +6\beta +C\right) +3\left( 9\beta ^{2}-\alpha
\beta +2C\beta +\mu \right) \left( \alpha +\beta \right) \left( 13\beta
+2C\right) =0  \label{nksoleqab2}
\end{equation}%
Therefore, $\alpha $ and $\beta $ satisfy the polynomial equations (\ref%
{nksoleql2}) and (\ref{nksoleqab2}). By explicit calculations (using \emph{%
Maple}) it can be shown that $\alpha $ and $\beta $ must be constants that
depend on $C,l$ and $\mu .$ However, it means that in equation (\ref%
{nksoleqa}), either $\alpha =0$ or $l=0.$ Suppose $\alpha =0$. Then, from (%
\ref{nksoleql}) we get that 
\[
\beta ^{2}=\frac{\mu }{3} 
\]%
So $\beta ^{\prime }=0,$ and thus from (\ref{nksoleqb}), we get 
\[
\left( 9\beta ^{2}+2C\beta +\mu \right) \beta =0 
\]%
So either $\beta =0,$ which is a case we already covered in (\ref{alphabet0}%
), or 
\begin{equation}
\beta =-\frac{1}{6}C,\ \mu =\frac{1}{12}C^{2}  \label{betamusol}
\end{equation}%
Note that in these cases $l$ is an arbitrary constant.
\end{description}

The other possibility is that $l=0$. So suppose $\alpha \neq 0$. In this
case, $\alpha $ and $\beta $ have to satisfy (\ref{nksoleql2}) and 
\begin{equation}
\left( 9\beta ^{2}-\alpha \beta +2C\beta +\mu \right) \left( \alpha +\beta
\right) =0  \label{nksoleqab3}
\end{equation}%
Then either $\beta =-\alpha $, where $\alpha $ satisfies 
\[
10\alpha ^{2}-2C\alpha +\mu =0, 
\]%
which gives solution 4, or 
\begin{equation}
9\beta ^{2}-\alpha \beta +2C\beta +\mu =0.  \label{nksoleqab4}
\end{equation}%
Note that the equations (\ref{nksoleql2}) and (\ref{nksoleqab4}) are
equivalent if $\beta =-\alpha $, but we already considered this case.
Suppose $\beta \neq 0$. Then, from (\ref{nksoleqab4}), 
\[
\alpha =\frac{1}{\beta }\left( \mu +2C\beta +9\beta ^{2}\right) 
\]%
Using this we find that (\ref{nksoleql2}) simplifies to 
\[
\left( \mu -3\beta ^{2}\right) \left( 10\beta ^{2}+2C\beta +\mu \right) =0 
\]%
and hence%
\[
\alpha \left( \mu -3\beta ^{2}\right) \left( \alpha +\beta \right) =0. 
\]%
We have already considered the cases $\alpha =0$ and $\alpha +\beta =0$,
hence $\beta ^{2}=\frac{\mu }{3}$. Thus we get solution 5. If however, $%
\beta =0$, then equation (\ref{nksoleqab4}) forces $\mu =0$, and hence from (%
\ref{nksoleql2}) $\alpha =0$ or $\alpha =2C$. These cases are however
already covered. Therefore this exhausts all the possible solutions.
\end{proof}

\begin{remark}
Recall from (\ref{torscompabc}) that the $\tau _{27}$ component of the
torsion is proportional to $\alpha +\beta $. Therefore, in Proposition \ref%
{Propconstlsol}, the solution 5 always has $\tau _{27}=0,$ since $\alpha
+\beta =0.$ Therefore, this solution corresponds to a nearly parallel $G_{2}$
structure with only the $\tau _{1}$ torsion component being non-zero. Since
in this solution the only restriction on $\mu $ is that $\mu \leq \frac{C^{2}%
}{10}$ we may have solutions of different kinds - shrinking solutions ($\mu $
negative), expanding solutions ($\mu $ positive), and steady solutions ($\mu
=0$). Similarly, in solution 5, we can get $\alpha +\beta =0$ if $C$ is
negative and 
\[
\mu =\frac{12}{169}C^{2}. 
\]%
So for this value of $\mu $ we obtain another nearly parallel solution.
However this time, these are all expanding solutions (except the trivial
steady case when $C=0$). Moreover, by fixing 
\[
\mu =\frac{1}{3}C^{2} 
\]%
we obtain $\alpha -6\beta =0.$ This gives $\tau _{1}=0$, and a non-zero $%
\tau _{27}$ - hence this soliton solution is a $G_{2}$-structure of pure
type $\mathbf{27}.$
\end{remark}

A special simple solution is $\alpha =\beta =0$. Note that in this case, the
torsion vanishes, and we have a torsion-free $G_{2}$-structure. Then, the
first equation just becomes 
\begin{equation}
l^{\prime }=-\mu
\end{equation}%
From the definition of $\beta ,$ we find that $\sin \theta =0$, so $\theta
=0 $ or $\pi $. However from $\gamma =0$, we also have $h^{\prime }=-\lambda
\cos \theta .$ Thus we have the following solutions: 
\begin{equation}
\left\{ 
\begin{array}{c}
\theta =0,\ \ h=-\lambda r+h_{0},\ l=-\mu r+l_{0} \\ 
\theta =\pi ,\ \ h=\lambda r+h_{0},\ l=-\mu r+l_{0}%
\end{array}%
\right.  \label{alphabet0}
\end{equation}%
These are precisely the solutions also obtained in \cite%
{KarigiannisMcKayTsui}.

Consider now another special case where $l$ is constant. It then turns out
that if $l$ is constant, then necessarily, either $l=0$ or $\mu $ has take a
particular value that depends on $C$. These cases then cover all the
remaining exact solutions of the corresponding system that were found in 
\cite{KarigiannisMcKayTsui}, as well as additional solutions.

Proposition \ref{Propconstlsol} gives us the critical points of the
equations (\ref{nksoleq}), however it can be easily seen that linearizations
at these critical points are degenerate, and therefore do not provide much
information about the full system.

\bibliographystyle{jhep-a}
\bibliography{refs2}

\end{document}

%% file: epsframe.tex
\def\Qlb#1{#1}

\def\Qcb#1{#1} 

\def\FRAME#1#2#3#4#5#6#7#8
{
 \begin{figure}[ptbh]
 \begin{center}
 \includegraphics[height=#3]{#7}
 \caption{#5}
 \label{#6}
 \end{center}
 \end{figure}
}

